\documentclass[11pt]{article}
\usepackage{amsthm,amsfonts,amsmath,amssymb}
\usepackage{pxfonts}
\usepackage{euscript}
\usepackage{hyperref}
\usepackage{graphics}
\usepackage{epstopdf} % for pdflatex
\usepackage{xypic}
\usepackage[all,arrow,matrix]{xy}

\usepackage{graphicx}

\usepackage{sectsty}
\sectionfont{\large}

\newcommand\void[1]       {}

\newtheorem{thm}{Theorem}
\newtheorem{defn}[thm]{Definition}
\newtheorem{prop}[thm]{Proposition}
\newtheorem{cor}[thm]{Corollary}
\newtheorem{rema}[thm]{Remark}
\newtheorem{lem}[thm]{Lemma}

\newtheorem{conj}[thm]{Conjecture}
\newtheorem{expl}[thm]{Example}

\numberwithin{equation}{section}
\numberwithin{thm}{section}

\newcommand\nn             {\nonumber \\}

\newcommand\be            {\begin{equation}}
\newcommand\ee            {\end{equation}}
\newcommand\bea           {\begin{eqnarray}}
\newcommand\eea         {\end{eqnarray}}
\newcommand\bnu          {\begin{enumerate}}
\newcommand\enu          {\end{enumerate}}

\newcommand\id            {\mathrm{id}}
\newcommand\op          {\mathrm{op}}
\newcommand\ob          {\mathrm{Ob}}
\newcommand\ev          {\mathrm{ev}}

\newcommand\aut      {\mathrm{Aut}}

\newcommand\rep     {\EuScript{R}\mathrm{ep}}
\newcommand\fun     {\mathrm{Fun}}
\newcommand\Fun    {\EuScript{F}\mathrm{un}}

\newcommand\one    {\mathbf{1}}

\newcommand{\rev} {\mathrm{rev}}

\newcommand{\pf}{\begin{proof}}
\newcommand{\epf}{\end{proof}}

\newcommand\Cb            {\mathbb{C}}

\newcommand\Rb            {\mathbb{R}}
\newcommand\Zb            {\mathbb{Z}}

\newcommand\EA           {\EuScript{A}}
\newcommand\EB           {\EuScript{B}}
\newcommand\EC           {\EuScript{C}}
\newcommand\ED           {\EuScript{D}}
\newcommand\EE          {\EuScript{E}}
\newcommand\EF          {\EuScript{F}}

\newcommand\EI           {\EuScript{I}}
\newcommand\EK         {\EuScript{K}}
\newcommand\EL          {\EuScript{L}}
\newcommand\EM          {\EuScript{M}}
\newcommand\EN         {\EuScript{N}}

\newcommand\EP         {\EuScript{P}}

\newcommand\EW        {\EuScript{W}}

\newcommand\fr            {fr}
\newcommand\Rep        {\EuScript{R}\mathrm{ep}}
\newcommand\vect         {\mathrm{Vec}}
\newcommand\svect       {s\mathrm{Vec}}
\newcommand\fpdim       {\mathrm{FPdim}}

\newcommand\bfce        {$\mathrm{BFC}_{/\EE}$}

\newcommand\nbfce     {$\mathrm{NBFC}_{/\EE}$}

\newcommand\mtce       {$\mathrm{UMTC}_{/\EE}$}
\newcommand\ubfc        {$\mathrm{UBFC}$}
\newcommand\ubfce        {$\mathrm{UBFC}_{/\EE}$}
\newcommand\umtc       {$\mathrm{UMTC}$}
\newcommand\umtce       {$\mathrm{UMTC}_{/\EE}$}

\begin{document}

\begin{center} \LARGE
Modular extensions of unitary braided fusion categories\\
and 2+1D topological/SPT orders with symmetries
\end{center}

\vskip 2em
\begin{center}
{\large
Tian Lan$^{a,b}$, \, Liang Kong$^{c,d}$,\,  
Xiao-Gang Wen$^{e,a}$,
~\footnote{Emails:
{\tt tlan@perimeterinstitute.ca, kong.fan.liang@gmail.com, xgwen@mit.edu}}}
\\[2em]
$^a$ Perimeter Institute for Theoretical Physics\\
Waterloo, Ontario N2L 2Y5, Canada
\\[1em]
$^b$ Department of Physics and Astronomy\\
  University of Waterloo, Waterloo, Ontario N2L 3G1, Canada
\\[1em]
$^c$ Department of Mathematics and Statistics\\
University of New Hampshire, Durham, NH 03824, USA
\\[1em]
$^d$ Center of Mathematical Sciences and Applications\\ 
Harvard University, Cambridge, MA 02138, USA
\\[1em] 
$^e$ Department of Physics, Massachusetts Institute of Technology\\
Cambridge, Massachusetts 02139, USA
\end{center}

\vskip 4em

\begin{abstract}
%The fusion-braiding properties of particles within a 2-dimensional open disk can be described by unitary modular tensor categories (UMTC). In the presence of a finite bosonic or fermionic symmetry, we explain that they can be described by a UMTC $\EC$ over a symmetric fusion category (SFC) $\EE$ (i.e. a \umtce), where $\EE$ describes local particles that carry charges valued in $\EE$. 
A finite bosonic or fermionic symmetry can be described uniquely by a symmetric fusion category $\EE$. In this work, we propose that 2+1D topological/SPT orders with a fixed finite symmetry $\EE$ are classified, up to $E_8$ quantum Hall states, by the unitary modular tensor categories $\EC$ over $\EE$ and the modular extensions of each $\EC$. In the case $\EC=\EE$, we prove that the set $\EM_{ext}(\EE)$ of all modular extensions of $\EE$ has a natural structure of a finite abelian group. We also prove that the set $\EM_{ext}(\EC)$ of all modular extensions of $\EC$, if not empty, is equipped with a natural $\EM_{ext}(\EE)$-action that is free and transitive. Namely, the set $\EM_{ext}(\EC)$ is an $\EM_{ext}(\EE)$-torsor. As special cases, we explain in details how the group $\EM_{ext}(\EE)$ recovers the well-known group-cohomology classification of the 2+1D bosonic SPT orders and Kitaev's 16 fold ways. We also discuss briefly the behavior of the group $\EM_{ext}(\EE)$ under the symmetry-breaking processes and its relation to Witt groups. 
\end{abstract}

\void{
\begin{abstract}
The fusion-braiding properties of particles within a 2-dimensional open disk can be
described by unitary modular tensor categories (UMTC). We explain that, in the presence
of a finite bosonic or fermionic symmetry, these properties can be described by a UMTC
$\EC$ over a symmetric fusion category (SFC) $\EE$ (i.e. a \umtce), where $\EE$ describes local
particles that carry charges valued in $\EE$. 
%and can uniquely describes a finite bosonic or fermionic symmetry.  
In general, the fusion-braiding properties described by a \umtce, however, 
do not satisfy the so-called anomaly-free principle, that every particle can be remotely
detected by some particles via double braidings.  In this work, we propose that the anomaly-free
fusion-braiding properties are described by a \umtce ~that allows
modular extensions. Different modular extensions for a fixed $\EC$ have different physical meanings. We propose that 2+1D topological orders with a fixed finite symmetry $\EE$ are
classified, up to $E_8$ quantum Hall states, by \umtce's $\EC$ together with the modular extensions of $\EC$.  In the case $\EC=\EE$, we prove that the set $\EM_{ext}(\EE)$ of all modular extensions of $\EE$ has a natural structure of an abelian group. We also prove that the set
$\EM_{ext}(\EC)$ of all modular extensions of $\EC$ form an
$\EM_{ext}(\EE)$-torsor. As special cases, we explain in details how the groups
of modular extensions recover the classification of the 2+1D bosonic SPT orders
by group cohomology and that of the fermonic ones by Kitaev's 16 fold ways. We
also discuss briefly the behavior of the group $\EM_{ext}(\EE)$ and the
$\EM_{ext}(\EE)$-torsors under the symmetry-breaking processes. 
\end{abstract}
}

\newpage
\tableofcontents

\vspace{1cm}

\section{Introduction}

In this work, we prove some mathematical results about (unitary) braided fusion
categories that are completely motivated by the physics of topological orders
with symmetries (see Sec.\,\ref{sec:physics}). More physical discussion of the same subject will appear in a companioned physics paper \cite{lkw2}. In this introduction section, we try to keep the physics to the minimum and mainly focus on introducing our main results in mathematics. %Mathematically oriented readers can skip Sec.\,\ref{sec:physics} completely. 

\void{
It is well known that the consistent fusion-braiding properties of particles on a
2-dimensional open disk are described by unitary modular tensor
categories (UMTC) \cite{MS,BK,kitaev}.  The UMTC satisfies a so called
anomaly-free principle \cite{Levin,kong-wen}: every particles can be probed
remotely be a detector, which is another particle, via double braiding . 

In the presence of symmetry, there will be local particles that carry
representations of the symmetry group.  Those particles form a symmetric fusion
category (SFC) $\EE$.  In fact, a finite bosonic or fermonic symmetry, i.e. a
finite group $G$ or $(G,z)$ (see Sec.\,\ref{sec:physics}), can be characterized
by the SFC $\EE$ uniquely up to braided monoidal equivalences.  So to describe
the fusion and braiding of particles with symmetry, we need to include the
particles described by SFC $\EE$ which encode the symmetry.  Thus, we propose
that the fusion and braiding of particles with symmetry are described by
non-degenerate braided fusion category $\EC$ over SFC $\EE$.  

However, such a fusion and braiding described by $\EC$ over $\EE$ may not
satisfy the anomaly-free principle (and may not be consistent physically).  We
propose that an anomaly-free fusion and braiding are described by a $\EC$ over
$\EE$ that has modular extensions.  Such an anomaly-free fusion and braiding
are not only consistent on an open disk, they also consistent on any
2-dimensional manifolds (which may contain symmetry twists).

In general, a given  $\EC$ over $\EE$ may have several  modular extensions.
Amazingly, those different modular extensions have physical meanings:
}

\medskip
A finite bosonic or fermonic symmetry, i.e. a finite group $G$ or $(G,z)$ (see Sec.\,\ref{sec:physics}), is uniquely determined by a symmetric fusion category (SFC) $\EE$ up to braided monoidal equivalences. In Sec.\,\ref{sec:physics}, we propose %(see Conjecture\,\ref{conj:set}) 
that 2+1D topological orders \cite{Wtop} and symmetry protected trivial (SPT) orders \cite{gu-wen,cglw}
with an on-site symmetry $\EE$ are classified, up to $E_8$ quantum Hall states, by the equivalence classes of the triples $(\EC, \EM, \iota_\EM)$, which are explained below.  
\bnu

\item $\EC$ is a unitary modular tensor category $\EC$ over $\EE$, or a \umtce, which is defined by a unitary braided fusion category $\EC$ such that its M\"{u}ger center is $\EE$ (see Def.\,\ref{def:umtce}). Physically, the \umtce ~$\EC$ describes all the excitations in the bulk of the associated topological states. The bulk excitations, in general, are not enough to uniquely determine the topological order. %because it misses some information of the edge theories, which are usually not unique.  

\item $\EM$ is a unitary modular tensor category (UMTC), and $\iota_\EM: \EC \hookrightarrow \EM$ is a braided full embedding such that the M\"{u}ger centralizer of $\EE$ in $\EM$ is $\EC$. The pair $(\EM, \iota_\EM)$ is called a modular extension of $\EC$, a notion which was first introduced by M\"{u}ger \cite{mueger1} (see Remark\,\ref{rema:minimal}). Physically, a modular extension of $\EC$ amounts to a categorical way of gauging the categorical symmetry $\EE$ in $\EC$ (see Sec.\,\ref{sec:physics}).
%\footnote{What we call modular extension was called ``a minimal modular extension'' in \cite[Conjecture\,5.2]{mueger1}. For simplicity, we drop the term ``minimal" until the last section.}.  
\enu

We denote the set of equivalence classes of the modular extensions of a fixed \umtce ~$\EC$ by $\EM_{ext}(\EC)$ (see Def.\,\ref{def:eq-mext}). The simplest example of \umtce ~is just the SFC $\EE$ itself. The modular extensions of $\EE$ always exist. For example, $(Z(\EE), \iota_0)$, where $Z(\EE)$ is the Drinfeld center of $\EE$ and $\iota_0:\EE \to Z(\EE)$ is the canonical full embedding, is a modular extension of $\EE$. For generic $\EC$, Drinfeld showed that the set $\EM_{ext}(\EC)$ can be empty \cite{drinfeld}. 
 %In this case, all the associated topological orders (with the same bulk excitations given by $\EE$) are classified, up to $E_8$ quantum Hall states, by the modular extensions of $\EE$. 
%It was generally believed that such topological orders are all invertible topological orders. 

\medskip
The main results of this work are summarized in the following theorem. 
\begin{thm}
The set $\EM_{ext}(\EE)$, together with a naturally defined multiplication $\boxtimes_\EE^{(-,-)}$ and the identity element $(Z(\EE), \iota_0)$, is a finite abelian group. For a \umtce ~$\EC$, the set $\EM_{ext}(\EC)$, if not empty, is naturally equipped with a free and transitive $\EM_{ext}(\EE)$-action, or equivalently, an $\EM_{ext}(\EE)$-torsor. 
\end{thm}

In Sec.\,\ref{sec:mext}, we give a detailed construction of the multiplication $\boxtimes_\EE^{(-,-)}$ (see Lemma\,\ref{lem:MEN}), which has a natural physical meaning. 
%and prove the identity properties in Cor.\,\ref{cor:MM=1}, \ref{rema:unit-of-group}. 
We prove the first half of above theorem in Thm.\,\ref{thm:group} and the second half in Thm\,\ref{thm:torsor}. In the end, we prove some results on the behavior of the group $\EM_{ext}(\EE)$ and the $\EM_{ext}(\EE)$-torsor $\EM_{ext}(\EC)$ under the symmetry-breaking processes, and explain a relation between the modular extensions and the Witt groups.
% \cite{dmno,dno}. 

\medskip
The layout of this paper is as follows: in Sec.\,\ref{sec:physics}, we explain our motivations from the physics of topological orders with symmetries; in Sec.\,\ref{sec:bfc}, \ref{sec:alg-bfc}, \ref{sec:sfc}, we recall some basic concepts in braided fusion categories, collect and prove some useful results, and set our notations;
%We also collect some useful results, which are, occasionally, generalized; 
in Sec.\,\ref{sec:bfc-over-e}, we review the notion of a braided fusion category over a symmetric fusion category; in Sec.\,\ref{sec:unitary}, we recall some results on unitarity; in Sec.\,\ref{sec:mext-def}, we recall the notion of a modular extension of a \umtce ~and prove a few useful results; in Sec.\,\ref{sec:mext-E-group}, we prove that the set $\EM_{ext}(\EE)$ is a finite abelian group; in Sec.\,\ref{sec:mext-repG}, \ref{sec:mext-svect}, we explain the relation between $\EM_{ext}(\rep(G))$ and the group-cohomology classification of 2+1D bosonic SPT orders and that between $\EM_{ext}(\svect)$ and Kitaev's 16 fold way; in Sec.\,\ref{sec:torsor}, we prove that the set $\EM_{ext}(\EC)$ is an $\EM_{ext}(\EE)$-torsor; In Sec.\,\ref{sec:sym-breaking}, we discuss the behavior of $\EM_{ext}(\EE)$ and $\EM_{ext}(\EC)$ under certain symmetry breaking processes; in Sec.\,\ref{sec:sym-breaking}, we discuss the relation between the modular extensions and the Witt groups; in Sec.\,\ref{sec:outlook}, we give conclusions and list a few open questions.

\vspace{0.4cm}
\noindent {\bf Acknowledgement}: We thank Pavel Etingof, Chenjie Wang and Zhenghan Wang for useful discussions, and Dmitri Nikshych for helping us to understand the $\mathrm{Aut}(G)$-action on the group cohomologies %(see Remark\,\ref{rema:dmitri}) 
and informing us the results in \cite{bnrw}. This research is supported by NSF Grant No. DMR-1506475, and NSFC 11274192. It is also supported by the John Templeton Foundation No.  39901.  Research at Perimeter Institute
is supported by the Government of Canada through Industry Canada and by the
Province of Ontario through the Ministry of Research.  LK is supported by the
Center of Mathematical Sciences and Applications at Harvard University.

\section{Physics of topological/SPT orders with symmetries} \label{sec:physics}

Topological phases of matters, also called topological orders \cite{Wtop}, are
motivated by the experimental discovery of fractional quantum Hall effects. It
was realized recently that such a topological order can be understood as a
gapped quantum liquid with long range entanglement, a notion which was
introduced in \cite{LRE,zeng-wen}. A bosonic (or fermionic) topological order
is called anomaly-free if it can be realized by a bosonic (or fermionic) local
Hamiltonian lattice model in the same dimension \cite{Wanom}, and is called
anomalous if otherwise \cite{Levin,kong-wen}. In this work, we are only interested in
anomaly-free topological orders (with or without symmetries). For simplicity,
by a topological order (with or without symmetries), we mean an anomaly-free
topological order throughout this work unless we specify otherwise. 

\medskip
It is well-known that 2+1D topological orders without symmetries are classified, up to $E_8$ quantum Hall states, by the categories of excitations in the bulk. A particle-like excitation is called a local excitation if it can be created/annihilated by local operators from the vacuum sector or its direct sums; it is called a non-trivial topological excitations if otherwise. The vacuum sector or its direct sums can be viewed as local excitations. All (particle-like) excitations can be fused and braided, thus form a unitary braided fusion category $\EC$. The vacuum sector corresponds to the tensor unit
$\one_\EC$ in $\EC$. Excitations that correspond to simple objects in $\EC$ are
called simple excitations. By the definition of a local
excitation, it must have trivial double braidings with all (topological or local)
excitations. In an anomaly-free theory, excitations should be able to detect
each other via double braidings. In particular, the only simple excitation that
has trivial double braidings with all excitations must be the vacuum.
Therefore, the only local excitations in a 2+1D topological order are the
vacuum and its direct sums. Categorically, this amounts to say that the unitary
braided fusion category $\EC$ must be non-degenerate. This gives us the
well-known fact that the category of excitations in the bulk of a 2+1D
topological order is given by a non-degenerate unitary braided fusion
category, or equivalently, a unitary modular tensor category (UMTC) (see
for example \cite{kitaev})\footnote{It was proposed in
\cite{kong-wen-zheng} that all bulk excitations, including string-like excitations, in a potentially anomalous 2+1D topological order are
described by a unitary fusion 2-category.}. Note that an $E_8$ quantum Hall state contains no
non-trivial topological excitations in the bulk. But it is non-trivial because
it cannot be changed to a trivial phase via local unitary transformations.
Moreover, it has gapless chiral edge modes that carries the central charge
$c=8$. By the term ``$E_8$ quantum Hall states'', we mean stacking finite
number of layers of $E_8$ quantum Hall states. Since a UMTC defines the central
charge only modulo 8, all 2+1D topological orders are classified by UMTC's
together with the central charges, i.e. by pairs $(\EC, c)$.

\medskip
In this work, we study long-range entangled topological orders with
symmetries.  In the presence of symmetries, even product states with short
range entanglement can belong to different phases, and those phases are called
symmetry protected trivial (SPT) orders.  So in fact, we will study both topological orders and SPT orders (see for example \cite{Wtoprev,gu-wen,srfl,kitaev2,LRE,cglw,freed-moore,bbcw}
and references therein for this vast and fast growing topic). In particular,
we would like to find a categorical classification (up to $E_8$ quantum Hall states) of topological/SPT orders
with an on-site finite symmetry. A finite
symmetry is mathematically described by a finite group $G$. It is called a
fermonic symmetry if $G$ contains the fermonic parity transformation $z$ in
$G$.  Mathematically, the fermonic parity transformation $z\in G$ is a central
element in $G$ (i.e.  $zg=gz, \forall g\in G$) and $z^2=1$. We denote the
fermonic symmetry by the pair $(G,z)$. If $G$ does not contain the fermonic parity transformation, it is called a bosonic symmetry, denoted by $G$ alone. We
denote the category of representations of $G$ by $\rep(G)$. It is a symmetric
fusion category (see Sec.\,\ref{sec:sfc}). For the fermonic symmetry $(G,z)$,
we require that a $G$-representation with $z$ acting as $-1$ (i.e. a fermion)
should braid as a fermion. This gives us a new symmetric fusion category
$\rep(G,z)$, which is the same fusion category as $\rep(G)$ but equipped with a
different braiding (see Sec.\,\ref{sec:sfc}).  When $G=\Zb_2$, the category
$\rep(\Zb_2,z)$ is just the category of super vector spaces $\svect$, i.e.
$\rep(\Zb_2,z)=\svect$. It is known from Deligne's result \cite{deligne} that a symmetric
fusion category $\EE$ is equivalent to either $\rep(G)$ or $\rep(G,z)$ for some
finite group $G$ and a central involutive element $z\in G$.  Therefore, we can define a finite bosonic/fermionic symmetry
simply by a symmetric fusion category $\EE$.

In a 2+1D topological/SPT order with an on-site symmetry $G$ or $(G,z)$, it is clear that the excitations in the bulk still forms a braided fusion category $\EC$. But local excitations in the bulk can carry symmetry charges, which are given by the representations of $G$ or $(G,z)$. In other words, if $\EE$ denotes $\rep(G)$ or $\rep(G,z)$, then $\EC$ must contain $\EE$ as a fusion subcategory (see Sec.\,\ref{sec:bfc} for a definition). Moreover, since these local excitations must have trivial double braidings with all excitations, mathematically, it means that $\EE$ is contained in the M\"{u}ger center of $\EC$ (\cite{mueger1} and see also Sec.\,\ref{sec:bfc}), which is denoted by $\EC'$. Such $\EC$ is called a unitary braided fusion category over $\EE$, or a \ubfce. Although these local excitations in $\EE$ are not detectable by double braidings, they are not anomalous because they are ``protected'' by the symmetry. For anomaly-free topological orders, we can not allow any local excitations that are not protected by the symmetry. In other words, excitations that have trivial double braidings with all excitations in $\EC$ must be those in $\EE$. Mathematically, it means that $\EC'=\EE$. Such a \ubfce ~$\EC$ is called a non-degenerate \ubfce, or a unitary modular tensor category over $\EE$ (a \umtce). Therefore, the excitations in the bulk of a 2+1D topological order with the symmetry $\EE$ must be given by a \umtce. In the simplest case, there is no non-trivial topological excitation in the bulk, i.e. $\EC=\EE$. %In this case, the associated (usually non-unique) topological phases must be trivial (with only short ranged entanglement) but has ``non-anomalous'' local excitations that are protected by symmetry. 

Different from no-symmetry cases, the bulk excitations do not uniquely fix the
associated topological orders up to $E_8$ quantum Hall states. We give two sets of examples. 
\bnu

\item For topological orders with a finite bosonic symmetry $G$ and only symmetry protected local bulk excitations $\EC=\rep(G)$, it is known that there are different SPT orders classified by 3-cocycles in $H^3(G, U(1))$ \cite{cglw}. 

\item For the topological orders with only fermonic parity symmetry $(\Zb_2, z)$ and only symmetry-protected local bulk excitations $\EC=\rep(\Zb_2,z)$, it is known that there are different gapless
chiral edge states classified by Kitaev's 16 fold ways \cite{kitaev}. These 16 phases are
generated by the $p+\mathrm{i} p$ superconductor state
with central charge $c=1/2$ (via stacking operations).  They are different topological orders despite
they have the same category of bulk excitations. 

\enu
Therefore, what we need is to add more data to $\EC$ such that they are able to distinguish topological/SPT orders associated to the same bulk excitations up to $E_8$
states. 

%the topological orders with a bosonic symmetry $G$ and a trivial bulk
%$\EC=\rep(G)$ are symmetry protected topological orders and was shown to be
%classified by $H^3(G, U(1))$. For 

In physics, many systems with hidden degrees of freedom protected by symmetries
can be detected by gauging the symmetry. Motivated by an idea of gauging \cite{LG}, we proposed in \cite{lkw1} a tensor-categorical way of gauging the categorical symmetry $\EE$ by adding more particles to the set of particles in a \umtce ~$\EC$ such that each of them has
non-trivial double braidings with at least one of the local excitations in
$\EE$ (see Remark\,\ref{rema:non-minimal}). This categorical gauging process is complete only if every excitation in
$\EE$ has non-trivial double braidings with at least one newly added particle,
and all the newly added particles, together with old ones in $\EC$, form a
closed and consistent anyon system in the sense that it describes the
bulk excitations of a new 2+1D anomaly-free topological order without symmetry.
Mathematically, a complete categorical gauging process just amounts to find a
UMTC $\EM$ equipped with a braided full embedding $\iota_\EM: \EC
\hookrightarrow \EM$ such that the M\"{u}ger centralizer of $\EE$ in $\EM$,
denoted by $\EE'|_\EM$, is $\EC$. Such a pair $(\EM, \iota_\EM)$ is called a modular extension of $\EC$, a notion which was first introduced by M\"{u}ger in \cite{mueger1} (see also Def.\,\ref{def:mext}). We explain in detail in Sec.\,\ref{sec:mext-repG} how modular extensions of $\rep(G)$ recover the group-cohomology classification of bosonic SPT orders, and in Sec.\,\ref{sec:mext-svect}, we review the well-known 16 modular extensions of $\rep(\Zb_2,z)$ (also known as Kitaev's 16 fold ways). For generic $\EE$, we propose that the modular extensions of $\EE$ with central charge $c=0$ ($\mathrm{mod}\,\, 8$) classify all the 2+1D bosonic/fermionic SPT orders. For general $\EC$, we believe that the modular extensions of $\EC$, if exists, classify, up to $E_8$ quantum Hall states, all the topological orders with the same bulk excitations $\EC$. Such topological orders will be called symmetry enriched topological (SET) orders over $\EC$. 

\begin{rema} \label{rema:non-minimal} {\rm
Requiring each newly added particle to have non-trivial braiding with at least one particle in $\EE$, i.e. $\EE'|_\EM=\EC$, is a strong condition. For a given \umtce ~$\EC$, gauging processes satisfying this strong condition might not exist at all \cite{drinfeld}. In this case, one might want to relax this condition. But it is not yet clear to us what the extensions without satisfying this strong condition represent in physics (see also Remark\,\ref{rema:minimal}). %According to M\"{u}ger \cite{mueger1}, any pair $(\EM, \iota_\EM: \EC \hookrightarrow \EM)$, where $\EM$ is a UMTC and $\iota_\EM$ a braided full embedding, is called a modular extension of $\EC$; and is called a {\it minimal modular extension} of $\EC$ if $\EE'|_\EM=\EC$ in addition. Since we only study minimal modular extensions in this work. We drop the term ``minimal'' until the last section. 
}
\end{rema}

In summary, we have proposed the following conjecture on the classification of 2+1D topological/SPT orders with symmetries. 
\begin{conj} \label{conj:set}
2+1D topological/SPT orders with the symmetry $\EE$ are classified, up to $E_8$ quantum Hall states, by the equivalence classes (see Remark\,\ref{rema:eq-relation}) of triples $(\EC, \EM, \iota_{\EM})$, where $\EC$ is a \umtce ~and the pair $(\EM, \iota_\EM)$ is a modular extension of $\EC$. 
In particular, 2+1D SPT orders with the symmetry $\EE$ are classified by triples
$(\EE, \EM, \iota_{\EM})$ such that $\EM$ has a zero ($\mathrm{mod}\,\, 8$) central charge.
\end{conj}

In this work, we prove that the set $\EM_{ext}(\EE)$ of modular extensions of
$\EE$ form an abelian group with the multiplication $\boxtimes_\EE^{(-,-)}$
(defined in Lemma\,\ref{lem:MEN}) corresponding to first stacking two layers of systems then breaking the symmetry from $\EE\boxtimes \EE$ to $\EE$ (causing no phase transition \cite{lkw2}). Therefore, $\boxtimes_\EE^{(-,-)}$ is the correct physical stacking operation of two layers of topological/SPT phases with symmetry $\EE$. Moreover, we prove that the group $\EM_{ext}(\rep(G))$ coincides with the group structure on $H^3(G,U(1))$. We also prove that $\EM_{ext}(\svect)\simeq \Zb_{16}$ as groups. 
%It is consistent with the fact that stacking 16 layers of $p+\mathrm{i} p$ superconductor states gives exactly an $E_8$ quantum Hall state \cite{lkw2}. 

\begin{rema} \label{rema:eq-relation}  {\rm
%According to Conjecture\,\ref{conj:set}, the SET orders over a \umtce ~$\EC$, if exist, are classified (up to $E_8$ states) by $\EM_{ext}(\EC)$. 
We prove in Thm\,\ref{thm:torsor} that the set $\EM_{ext}(\EC)$, if not empty, is equipped with a natural $\EM_{ext}(\EE)$-action that is free and transitive. Namely, $\EM_{ext}(\EC)$ is a $\EM_{ext}(\EE)$-torsor. It is also equipped with a natural $\underline{\mathrm{Aut}}(\EC)$-action (see Remark.\,\ref{rema:aut-C} and \ref{rema:aut-C-2}). 
%This leads us to something that is completely unknown in physics before. 
Using the naturally defined equivalence relation among the triples 
$(\EC, \EM, \iota_\EM)$ (see a precise definition in \cite{lkw2}), Conjecture\,\ref{conj:set} says that the SET orders over $\EC$ are classified, up to $E_8$ states, by the quotient set $\EM_{ext}(\EC)/\underline{\mathrm{Aut}}(\EC)$ (see \cite{lkw2} for more details). 
%This issue is irrelevant to the mathematical results in this work.
}
\end{rema}

%For a \umtce ~$\EC$, we prove that $\EM_{ext}(\EC)$ is equipped with a natural $\EM_{ext}(\EE)$-action. Moreover, this action is free and transitive. Namely, $\EM_{ext}(\EC)$ is an $\EM_{ext}(\EE)$-torsor. This result is completely new in physics. It says that SPT phases acts on SET phases freely and transitively, and the difference between two SET phases can be measured by a unique SPT phase. 

%%%%%%%%%%%%%%%%%%%%%%%%%%%%%%%%%

\void{
Gapped quantum liquid (GQL) phases --- equivalent classes of many-body systems under symmetric
(generalized) local unitary transformations,
including topological ordered phases (no symmetry), symmetry protected trivial
phases and symmetry enriched topological phases. 

Two perspectives ---
\begin{enumerate}
  \item Stacking of GQLs. Construct a two-layer system while keeping the same symmetry.
    GQLs with the same symmetry $\EE$ form a commutative monoid. In particular, there
    are symmetric invertible GQLs, that form an abelian group $iGQL_\EE$. Consider the
    group morphism, $c : iGQL_\EE \to \mathbb{Q}$ that
    takes the chiral central charge. Since there must be a minimal central
    charge $c^m_\EE$, we must have the image $\text{Im}\,c$ to be $c^m_\EE\Zb$.
    The kernel $\ker c$, i.e., symmetric invertible GQLs with zero central
    charge, are the
    symmetry protected trivial (SPT) phases. And we believe that this group
    splits: $iGQL_\EE= \ker c\times (c^m_\EE\Zb)\cong \ker c \times \Zb$.
    Invertible GQLs with non-zero central charge and non-invertible GQLs
    are considered to possess non-trivial topological orders.
  \item Braiding statistics of bulk topological excitations. In general bulk
    excitations form a unitary braided fusion category (UBFC).
    Non-degenerate UBFCs seem to be a complete
    classification of topological orders with no symmetry.
    When there is symmetry, the local
    excitations form a symmetric category $\EE$. So all the invertible GQLs
    with the same symmetry will have the same bulk excitations $\EE$.

    However, it turns out that this perspective can also distinguish different
    invertible GQLs, up to the $E_8$ states (invertible GQLs with no symmetry),
    by studying their modular extensions.

    More generally, we propose a classification of symmetric GQLs, in terms of this tensor
    categorical language.
    The symmetry is characterized by the symmetric category $\EE$. A generic
    GQL phase with symmetry $\EE$ is characterize by a non-degenerate UBFC over
    $\EE$ (\umtce), together with its modular extension and the total central charge.
\end{enumerate}

This paper is devoted to: 1) introduce the notions of \umtce and modular
extensions; 2) study how the stacking of GQLs is carried out in the categorical
language.

We introduce a tensor product $\boxtimes_\EE$ between \umtce's as well as their
modular extensions. As expected, they do form a commutative monoid;
the modular extensions of $\EE$, $\EM_{ext}(\EE)$, do form an abelian group,
which corresponds to
the quotient $iGQL_{\EE}/8\Zb=\ker c \times (c^m_\EE\Zb/8\Zb)$ where $8\Zb$ is generated by the symmetric $E_8$
state.

Moreover, our construction leads to some non-trivial prediction: modular
extensions of a nontrivial \umtce $\EC$, $\EM_{ext}(\EC)$, if exist,
form a torsor over $\EM_{ext}(\EE)$.
The physical interpretation is that: 1) stack a GQL with a non-trivial invertible one,
the phase always changes, but the bulk excitations stay the same; 2) if two
GQLs have the same bulk excitations, they must differ by a unique invertible
GQL. These are the nicest case to expect. But how are the other possibilities
ruled out physically? (other possibilities: a GQL may absorb a non-trivial
invertible one but stay in the same phase, or, there exist certain two GQLs with
the same bulk excitations, but no way to obtain one from the other by stacking
any invertible ones.)

\bigskip

%%%%%%%%%%%%%%%%%%%%%%%%%%%%%%%%%%

The notion of gapped quantum liquids was introduced in \cite{zeng-wen}. 

There are two simple cases of bosonic gapped quantum liquids: 
\begin{enumerate}
  \item The topological ordered states with no symmetry. The most important
    feature of topological order is that they have (non-abelian) anyonic
    excitations. The fusion and braiding statistics of the anyons are the
    universal (topological) properties of topological orders. They are
    naturally described by the mathematical structure of unitary braided fusion
    categories(UBFC). When we are considering purely 2+1D phases (anomaly-free),
    the braiding non-degeneracy condition is imposed. Thus, bulk excitations of
    topological orders with no symmetry are classified by non-degenerate UBFCs,
    or unitary modular tensor categories (UMTC). We believe that UMTCs,
    together with the total central charges that determines edges states,
    completely  classify topological orders with no symmetry.
  \item There is a symmetry $G$, but no topological orders. In this case they
    are symmetry protected trivial (SPT) phases. There are no anyonic
    excitations in the bulk, only local excitations carrying group
    representations. But the edges states are still non-trivial. SPT phases are
    classified by group cohomology theory.
\end{enumerate}
Now comes two questions:
\begin{enumerate}
  \item What classifies general GQLs that may possess both topological orders
    and symmetry, or symmetry enriched topological orders (SET)?
  \item What about fermionic cases?
\end{enumerate}
We expect to combine UBFCs and group theory for a complete mathematical story of
GQLs. It it natural to consider the automorphisms between UBFCs. This leads to
the approach of $G$-crossed UBFCs. However, it works for bosonic cases, but not
fermionic cases.

In this work, instead, we focus on the group representations, which are carried
by the local excitations of symmetric phases. Group representations are just
UBFCs with maximal braiding degeneracy, the symmetric fusion categories (SFC).
Together with those anyonic excitations, we propose that bulk excitations of
GQLs are classified by non-degenerate UBFCs over the SFC.

However, classifying the bulk excitations is not enough. The SPT phases with
the same symmetry $G$ all have the same bulk excitations $\Rep(G)$, but are
distinct phases, in that they have different edge states. For non-trivial
symmetry the central charge is not enough to characterize the edge states. We
need some additional structures. Ideally they can still be put in the framework
of UBFCs. For this we choose to ``gauge'' the symmetry, i.e., add symmetry
twists as extra
particles that braid non-trivially with original particles. We require that
the braiding of all these particles is non-degenerate. Thus, we obtain a UMTC
after gauging the symmetry. It is a \emph{modular extension} of the original
non-degenerate UBFC over the SFC.

Hopefully, non-degenerate UBFCs over SFCs, together with their modular
extensions and the total central charges, would give us a complete story for
the classification problem of GQLs. In this work, we show that this framework
agrees with all the known results, such as group cohomological classification of
bosonic SPTs and the $G$-crossed UMTC proposal.
We further obtain a new result: given a symmetry, the modular extensions of the
same UBFC form torsors over the invertible GQLs.% In other words: 1. stacking with a non-trivial invertible GQL, one always get a different phase; 2. any two SETs with the same bulk excitations always differ by stacking with certain invertible GQL.  It is not clear if this is true physically.

There exist topological ordered states with only trivial topological
excitations in the bulk, but non-trivial edge states. They are ``invertible''
under the stacking operation\cite{KW1458,F1478}. More generally, we define  
\begin{defn}
A GQL is invertible if its bulk topological excitations are all trivial.
\end{defn}
Consider some invertible GQLs with the same symmetry given by a symmetric fusion category 
$\EE$.  The bulk excitations of those invertible GQLs are the same which are described by the
same SFC $\EE$. Now the question is: How to distinguish those invertible GQLs
that may have different edge states?

First, we believe that invertible bosonic topological orders with no symmetry
are generated by the $E_8$ quantum Hall state (with central charge $c=8$) via
time-reversal and stacking, and form a $\Zb$ group. Stacking with an $E_8$ quantum Hall 
state only changes the central charge by $8$, and does not change the bulk excitations or
the symmetry. So the only datum we need to know to determine the invertible
bosonic topological order with no symmetry is the central charge $c$.  The
story is parallel for invertible fermionic topological orders with no symmetry,
which are believed to be generated by the $p+\mathrm{i} p$ superconductor state with
central charge $c=1/2$.

Secondly, invertible bosonic GQL's with symmetry are generated by bosonic SPT
states and invertible bosonic topological orders (i.e. $E_8$ states). It was well known that the SPT states with symmetry group given by $G$ are classified by the 3-cocycles in
$H^3(G,U(1))$ \cite{spt}. We briefly review this classification in this subsections. 
}

%%%%%%%%%%%%%%%%%%%%%%%%%%%%%%%%%%%

\section{Unitary braided fusion categories} \label{sec:umtc}

In this section, we briefly recall some basic elements of (unitary) braided fusion categories, collect and prove a few results that are useful later, and also set our notations. The ground field is always chosen to be the complex numbers $\Cb$.

\subsection{Braided fusion categories}  \label{sec:bfc}

%A monoidal category $\EC$ is equipped with a tensor product $\otimes$ and a tensor unit $\one_\EC$. The tensor product $\otimes: \EC \times \EC \to \EC$ is associative with the associativity isomorphisms $\alpha_{X,Y,Z}: \, X\otimes (Y \otimes Z) \xrightarrow{\simeq} (X\otimes Y) \otimes Z$ for $X,Y,Z \in \EC$ required to satisfy the pentagon relations. The unit isomorphisms $\one_\EC \otimes X \xrightarrow{l_X} X \xleftarrow{r_X} X\otimes \one_\EC$ are required to satisfy the triangle relations. If $\EC$ is also braided, then it equips with braidings, which are a family of isomorphisms $c_{X,Y}: X\otimes Y \xrightarrow{\simeq} Y\otimes X$, satisfying the hexagon relations. 

In this work, a category is called {\it finite} if it is equivalent to the
category of modules over a finite dimensional algebra $A$ over $\Cb$; it is
called {\it semisimple} if, in addition, $A$ is semisimple. A multi-fusion
category $\EC$ is a semisimple rigid monoidal category. In particular, it has
finitely many inequivalent simple objects, and is equipped with a rigid
monoidal structure, which includes the tensor product functor $\otimes: \EC
\times \EC \to \EC$, the tensor unit $\one_\EC \in \EC$,  an associator $\alpha_{x,y,z}: \, x\otimes (y \otimes z) \xrightarrow{\simeq} (x\otimes y) \otimes z$ for $x,y,z \in \EC$ satisfying the pentagon relations, and unit isomorphisms satisfying the triangle relations. We denote the right dual of an object $x$ as $x^\ast$ and the left dual as ${}^\ast x$. We denote the set of equivalence classes of simple objects by $O(\EC)$. We denote $\EC^\rev$ by the same category $\EC$ but equipped with the reversed tensor product $\otimes^{\rev}$, i.e. $a\otimes^{\rev} b=b\otimes a$. A multi-fusion category $\EC$ with a simple tensor unit $\one_\EC$ is called a {\it fusion category}. A fusion subcategory $\EB$ of $\EC$, denoted by $\EB\subset \EC$, is a full tensor subcategory such that if $x\in \EC$ is isomorphic to a direct summand of an object in $\EB$ then $x\in \EB$. In particular, $\EB$ is a fusion category and $O(\EB)\subset O(\EC)$.

%There are two different but related notions of dimension in a fusion category $\EC$: the quantum dimension and Frobenius-Perron dimension. 
\smallskip
Let $K_0(\EC)$ be the Grothendieck ring of a fusion category $\EC$. According to \cite[Sec.\,8]{eno2002}, there is a unique homomorphism $\fpdim: K_0(\EC) \to \Rb$ such that $\fpdim(x)\geq 0$ for all $x\in \EC$. Actually, $\fpdim(x) \geq 1$ for any non-zero object $x\in \EC$. $\fpdim(x)$ is called the Frobenius-Perron dimension of $x$. The Frobenius-Perron dimension of a fusion category $\EC$ is defined by $\fpdim(\EC)=\sum_{i\in O(\EC)} \fpdim(i)^2$. For a pivotal fusion category $\EC$, there is a different but related notion of dimension, which is called {\it quantum dimension} and denoted by $\dim(x)$ for $x\in \EC$. For a unitary fusion category (see Def.\,\ref{def:unitary}), the Frobenius-Perron dimensions coincide with the quantum dimensions \cite{eno2002}. 

The Frobenius-Perron dimension is very useful in determining whether a given monoidal functor is an equivalence. Consider a monoidal functor $F: \EC \to \ED$ between two fusion categories $\EC$ and $\ED$. We define the image of $F$, denoted by $\mathrm{Im}(F)$, by the smallest fusion subcategory of $\ED$ that contains $F(x), \forall x\in \EC$. By \cite[Prop.\,2.19]{eo}, if $F$ is injective (i.e. $F:\EC\to \mathrm{Im}(F)$ is an equivalence), then $\fpdim(\EC) \leq \fpdim(\ED)$ and the equality holds iff $F$ is a monoidal equivalence. By \cite[Prop.\,2.20]{eo}, if $F: \EC\to \ED$ is surjective (i.e. $\mathrm{Im}(F)=\ED$), then $\fpdim(\EC) \geq \fpdim(\ED)$ and the equality holds iff $F$ is a monoidal equivalence.

\medskip
A left module category $\EM$ over a fusion category $\EC$, or a $\EC$-module $\EM$, is a semisimple $\Cb$-linear category equipped with a $\EC$-action functor $\otimes: \EC \times \EM \to \EM$, which is $\Cb$-linear and right exact for each variables.  For two $\EC$-modules $\EM$ and $\EN$, a $\EC$-module functor $F: \EM \to \EN$ is a $\Cb$-linear functor, equipped with a natural isomorphism $F(c\otimes -) \simeq c\otimes F(-), \forall c\in \EC$, satisfying natural conditions. We denote the category of $\EC$-module functors by $\Fun_\EC(\EM, \EN)$. A right $\EC$-module is just left $\EC^\rev$-module. A $\EC$-$\ED$-bimodule is just a left $\EC\boxtimes\ED^\rev$-module. We denote the category of right exact $\EC$-$\ED$-bimodule functors from $\EM$ and $\EN$ by $\Fun_{\EC|\ED}(\EM,\EN)$. If $\EM$ is an indecomposable $\EC$-module, it is known that $\Fun_\EC(\EM,\EM)$ is also a fusion category.

\medskip
A braided fusion category $\EC$ is a fusion category equipped with a braiding $c_{x,y}: x\otimes y \to y\otimes x$ for $x,y\in \EC$. The category $\overline{\EC}$ is the same fusion category as $\EC$ but equipped with the braidings given by the anti-braidings in $\EC$. 

Given a fusion category $\EC$, there is a canonical braided fusion category $Z(\EC)$ associated to $\EC$, called the Drinfeld center of $\EC$. It consists of objects of pairs $(x, z_{x,-})$, where $z_{x,-}: x\otimes - \to -\otimes x$ (called a {\it half-braiding}) is a natural isomorphism satisfying some natural conditions. Another useful way to characterize the Drinfeld center is the category $\Fun_{\EC|\EC}(\EC,\EC)$ of $\EC$-bimodule functors from $\EC$ to $\EC$, i.e. $Z(\EC)=\Fun_{\EC|\EC}(\EC,\EC)$. We have $\fpdim(Z(\EC)) = \fpdim(\EC)^2$.
The forgetful functor $fr: Z(\EC) \to \EC$ is monoidal. We also have $Z(\EC)\simeq Z(\fun_\EC(\EM, \EM)^\rev)$ for an indecomposable $\EC$-module $\EM$. 

More generally, let $\EB$ be a fusion subcategory of a fusion category $\EC$. There is a notion of relative center $Z_\EB(\EC)$ consisting of pair $(x,z_{x,-})$, where $x\in \EC$ and the natural isomorphism $z_{x,-}: (x\otimes -)|_\EB \to (-\otimes x)|_\EB$ is the half-braiding. Another convenient way to characterize the relative center $Z_\EB(\EC)$ is the category $\Fun_{\EB|\EB}(\EB, \EC)$ of $\EB$-bimodule functors from $\EB$ to $\EC$, or equivalently, the category $\Fun_{\EB|\EC}(\EC, \EC)$ of right exact $\EB$-$\EC$-bimodule functors from $\EC$ to $\EC$ \cite{kz}. We have $\fpdim(Z_\EB(\EC))=\fpdim(\EB)\fpdim(\EC)$ and $Z(Z_\EB(\EC)) \simeq Z(\EB) \boxtimes Z(\EC)$ \cite{dgno2}.

\smallskip
Let $\EA$ be a fusion category and $\EC$ a braided fusion category. A monoidal functor $F: \EC \to \EA$ is called {\it central} if it factors uniquely through the forgetful functor $\fr: Z(\EA) \to \EA$. Namely, there is a unique braided monoidal functor $F_0: \EC \to Z(\EA)$ such that $\fr \circ F_0\simeq F$. A more direct way to characterize a central functor $F: \EC \to \EA$ is that there is a half-braiding $z_{c,x}: F(c)\otimes x \to x\otimes F(c)$ for $c\in \EC, x\in \EA$ such that $z_{c,F(d)}=F(c_{c,d})$ and $z_{c\otimes d, x}=z_{d,x}\circ z_{c,x}$, where $z_{d,x}\circ z_{c,x}:=(\id_c\otimes z_{d,x})\circ (z_{c,x}\otimes \id_d)$. 

\smallskip
Let $\EA$ be a full subcategory of a braided fusion category $\EC$. The M\"{u}ger centralizer of $\EA$ in $\EC$, denoted by $\EA'|_\EC$, is defined by the full subcategory consisting of objects $x \in \EC$ such that $c_{y,x}\circ c_{x,y} = \id_{x\otimes y}$ for all $y\in \EA$ \cite{mueger2}. Note that the M\"{u}ger centralizer $\EA'|_\EC$ is automatically a fusion subcategory of $\EC$ even if $\EA$ is not monoidal. We abbreviate $\EC'|_\EC$ to $\EC'$ and refer to it as the M\"{u}ger center of $\EC$. We have the following identity \cite{dgno2}:
\be \label{eq:fpdim-centralizer}
\fpdim(\EA)\fpdim(\EA'|_\EC) = \fpdim(\EC) \fpdim(\EA \cap \EC').
\ee
A braided fusion category $\EC$ is called {\it non-degenerate} if $\EC'=\vect$. 
%We abbreviate a non-degenerate braided fusion category to an \nbfc. 

If $\EC$ is braided, then there is a canonical braided monoidal functor $\EC \boxtimes \overline{\EC} \hookrightarrow Z(\EC)$ defined by $x\boxtimes y \mapsto (x, c_{x,-})\otimes (y,c_{-,y}^{-1})$. It is an equivalence iff $\EC$ is non-degenerate.

\subsection{Algebras in a braided fusion category}  \label{sec:alg-bfc}

Let $\EC$ be a fusion category. An {\it algebra in $\EC$} (or a {\it $\EC$-algebra}) is a triple $(A, m, \eta)$, where $A$ is an object in $\EC$, $m$ is a morphism $A\otimes A\to A$ and $\eta: \one \to A$ satisfying the identities 
$m \circ (m \otimes \id_A) \circ \alpha_{A,A,A} = m \circ (\id_A \otimes m)$ and 
$m \circ (\eta \otimes \id_A) = \id_A = m \circ (\id_A \otimes \eta)$. If $\EC$ is also braided, the $\EC$-algebra $A$ is called {\it commutative} if $m = m \circ c_{A,A}$. A right $A$-module, is a pair $(M, \mu_M)$, where $M$ is an object in $\EC$ and $\mu_M: M\otimes A \to M$ such that 
$\mu_M \circ (\id_M \otimes m) = \mu_M \circ (\mu_M \otimes \id_A) \circ \alpha_{M,A,A}$
and $\mu_M \circ (\id_M \otimes \eta) =\id_M$. The definition of a left $A$-module is similar.
%An $A$-$B$-bimodule is a triple $(M, \mu_M^L, \mu_M^R)$ such that $(M, \mu_M^L)$ is a left $A$-module and $(M, \mu_M^R)$ is a right $B$-module such that $$\mu_M^R \circ (\mu_M^L \otimes \id_B) \circ \alpha_{A,M,B} = \mu_M^L \circ (\id_A \otimes \mu_M^R). $$

We denote the category of right $A$-modules in $\EC$ by $\EC_A$. Let $\EB$ be a fusion subcategory of the fusion category $\EC$ and $A$ a $\EC$-algebra. We denote the maximal subobject of $A$ in $\EB$ by $A\cap \EB$. 
\begin{lem} \label{lem:embed}
Let $\EB$ be a fusion subcategory of a fusion category $\EC$ and $A$ a $\EC$-algebra such that $A\cap \EB=\one$. Then the functor $-\otimes A: \EB \to \EC_A$ is fully faithful. 
\end{lem}
\pf
This follows from the identities $\hom_{\EC_A}(x\otimes A, y\otimes A)\simeq \hom_\EC(x,y\otimes A) \simeq \hom_\EB(y^\ast\otimes x, A\cap \EB)\simeq \hom_\EB(y^\ast\otimes x, \one) \simeq \hom_\EB(x,y)$ for $x,y\in \EB$.  
\epf

A $\EC$-algebra $(A, m, \eta)$ is called {\it separable} if $m: A\otimes A \to A$ splits as a morphism of $A$-bimodule. Namely, there is an $A$-bimodule map $e: A\to A\otimes A$ such that $m \circ e = \id_A$. A separable algebra is called {\it connected} if $\dim \hom_\EC(\one, A) = 1$. If $\EC$ is braided, a commutative separable $\EC$-algebra is also called {\it \'{e}tale algebra} in \cite{dmno}. We abbreviate a connected commutative separable $\EC$-algebra to a {\it condensable algebra} for its physical meaning in anyon condensation \cite{kong}. If $\EC$ is non-degenerate, a condensable $\EC$-algebra $A$ is called {\it Lagrangian} if $\fpdim(A)^2=\fpdim(\EC)$.

\begin{rema} {\rm
Let $\EB$ be a fusion subcategory of fusion category $\EC$. If $A$ is a separable algebra in $\EC$, then $A\cap \EB$ is an separable algebra, which is just the internal hom $\underline{\hom}_\EB(A,A)$ \cite[Lemma\,3.2]{dno}. If $\EC$ is a braided and $A$ is \'{e}tale, then $A\cap \EB$ is an \'{e}tale algebra \cite[Cor.\,3.3]{dno}. 
}
\end{rema}

Let $\EC$ be a braided fusion category and $A$ a condensable $\EC$-algebra. The category $\EC_A$ of $A$-modules is a fusion category and we have the following identity: 
\be \label{eq:fpdim-A}
\fpdim(\EC_A) = \frac{\fpdim(\EC)}{\fpdim(A)}. 
\ee

If $\EC$ is braided and $A$ is commutative, a right $A$-module is called {\it local} if $\mu_M\circ c_{A,M} \circ c_{M,A}=\mu_M$. We denote the fusion subcategory of $\EC_A$ consisting of local $A$-modules in $\EC$ by $\EC_A^0$, which is actually a braided fusion category with the braidings inherited from those in $\EC$. If $\EC$ is non-degenerate, so is $\EC_A^0$ \cite{bek,ko}, and we have the following identities \cite{dmno}
\be \label{eq:ZC_A}
Z(\EC_A) \simeq \EC \boxtimes \overline{\EC_A^0} \quad\quad \mbox{and} \quad\quad \fpdim(\EC_A^0) = \frac{\fpdim(\EC)}{\fpdim(A)^2}.
\ee  
If, in addition, $A$ is Lagrangian, we have $\EC_A^0=\vect$. Moreover, if a condensable algebra $A$ contains a condensable subalgebra $B$, then $A$ is also a condensable algebra in the category $\EC_B^0$ and $\fpdim_{\EC_B^0}A=\frac{\fpdim(A)}{\fpdim(B)}$. We have $(\EC_B)_A\simeq \EC_A$ and $(\EC_B^0)_A^0\simeq \EC_A^0$ \cite{ffrs,da}. 

\medskip
It turns out that condensable algebras in a braided fusion category $\EC$ all arise in the following ways. 
\begin{thm} {\rm \cite{dmno}} \label{thm:dmno-1}
Let $\ED$ be a fusion category, $F: \EC \to \ED$ a central functor and $F^\vee: \ED \to \EC$ the right adjoint of $F$. The object $A=F^\vee(\one_\ED)$ has a canonical structure of condensable $\EC$-algebra, and the functor $F^\vee$ defines a monoidal equivalence $F^\vee: \mathrm{Im}(F) \to \EC_A$.  
\end{thm}

Let $\{ A\}$ be the full subcategory of $\EC$ consisting of a single object $A$. 
\begin{prop} \label{prop:embed}
Let $\EB$ be a fusion subcategory of a braided fusion category $\EC$ and $A$ a condensable $\EC$-algebra such that $A\cap \EB=\one$. The functor $-\otimes A: \{ A\}'|_\EC \cap \EB \to \EC_A^0$ is fully faithful and braided monoidal. 
\end{prop}
\pf
The functor $-\otimes A$ maps $\{ A\}'|_\EC$ into $\EC_A^0$ because $x\otimes A$ is a local $A$-module iff $x\in \{ A\}'|_\EC$. The fully-faithfulness follows from Lem.\,\ref{lem:embed}.  It is clearly braided monoidal. 
\epf

In this work, we are interested in condensable subalgebras of a condensable algebra in a non-degenerate braided fusion category $\EC$. Let $A$ be a condensable algebra in $\EC$. Let $L(\EC, A)$ be the lattice of condensable subalgebras of $A$ in $\EC$ and $\EL(\EC, A)$ the lattice of fusion subcategories of $\EC_A$ that contain $\EC_A^0$ as a fusion subcategory. The following theorem slightly generalizes Theorem~4.10 in \cite{dmno}.    
\begin{thm} \label{thm:dmno}
There is a canonical anti-isomorphism of lattices $\varphi: L(\EC, A)\simeq \EL(\EC, A)$. More precisely, for a condensable subalgebra $B$ of $A$, $\varphi(B)$ is defined by the image of the following functor 
$$
F_B: \EC_B^0 \boxtimes \overline{\EC_A^0} \xrightarrow{(-\otimes_B A)\boxtimes \id} \EC_A\boxtimes \overline{\EC_A^0} \xrightarrow{\otimes} \EC_A\, .
$$ 
Moreover, we have 
\bnu
\item $Z(\varphi(B)) \simeq \EC_B^0\boxtimes \overline{\EC_A^0}$ as non-degenerate braided fusion categories and the functor 
$Z(\varphi(B)) \simeq \EC_B^0 \boxtimes \overline{\EC_A^0} \xrightarrow{F_B} \varphi(B)$ 
coincides with the forgetful functor. 
\item For a fusion subcategory $\EB \subset \EC_A$, let $Z_\EB(\EC_A)$ be the relative center and $I: Z_\EB(\EA) \to Z(\EC_A)$ the right adjoint functor of the forgetful functor. Then we have $\varphi^{-1}(\EB)\simeq I(\one)$. 
\item $\fpdim(B)\fpdim(\varphi(B))=\fpdim(\EC_A)$. 
\enu
\end{thm}
\begin{proof}
When $A$ is a Lagrangian algebra, the result was proved in Theorem~4.10 in \cite{dmno} (with part 1 and 2 appeared in the proof of Theorem~4.10 in \cite{dmno}). 

In general cases, let $fr^\vee$ be the right adjoint of the forgetful functor $fr: Z(\EC_A)\to \EC_A$. The algebra $\tilde{A}=fr^\vee(\one)$ is Lagrangian, and we have $Z(\EC_A)_{\tilde{A}}\simeq \EC_A$ and $Z(\EC_A)_{\tilde{A}}^0\simeq\vect$. By Theorem~4.10 in \cite{dmno}, there is an anti-isomorphism $\phi$ from the lattice $L(Z(\EC_A), \tilde{A})$ of condensable subalgebras of $\tilde{A}$ to the lattice $\EL(Z(\EC_A), \tilde{A})$ of fusion subcategories of $\EC_A$. We have $Z(\EC_A)\simeq\EC\boxtimes \overline{\EC_A^0}$ as braided fusion categories, and the forgetful functor can be identified with the composed functor 
$$
Z(\EC_A)\simeq \EC\boxtimes \overline{\EC_A^0} \xrightarrow{(-\otimes A)\boxtimes \id } \EC_A\boxtimes \overline{\EC_A^0} \xrightarrow{\otimes} \EC_A.
$$ 
We have $A=\tilde{A} \cap (\EC \boxtimes \one)$. Namely, $A$ is a condensable subalgebra of $\tilde{A}$ in $Z(\EC_A)$. Let $B$ be a condensable subalgebra of $A$ in $\EC$. According to Theorem~4.10 in \cite{dmno}, the fusion subcategory $\phi(B\boxtimes \one)$ of $\EC_A$ can be identified with the image of the functor $F_B$. 
%$$Z(\EC_A)_{B\boxtimes \one}^0 = \EC_B^0 \boxtimes \overline{\EC_A^0} \xrightarrow{(-\otimes_B A)\boxtimes \id} \EC_A \boxtimes \EC_A^0 \xrightarrow{\otimes} \EC_A.  $$
Therefore, we have $\phi(B\boxtimes \one)=\varphi(B)$. Moreover, the image of the map $\phi$ restricted to the sub-lattice $L(\EC,A)=L(Z(\EC_A), A\boxtimes \one)$ of $L(Z(\EC_A), \tilde{A})$ is just the sub-lattice $\EL(\EC,A)$ of $\EL(Z(\EC_A), \tilde{A})$ because $\mathrm{Im}(F_A)=\EC_A^0$. Therefore, $\varphi=\phi|_{L(\EC,A)}$ is an anti-isomorphism from $L(\EC,A)$ to $\EL(\EC,A)$. The rest is clear.
\end{proof}

As an example, we give the following Proposition, which is an immediate consequence of Thm.\,\ref{thm:dmno} and Thm.\,\ref{thm:KO} and can be found in \cite[Example\,4.11]{dmno}. 
\begin{prop} \label{expl:G-H}
Take $\EC=Z(\vect_G^\omega)$ for a finite group $G$ and a 3-cocycle $\omega \in H^3(G,U(1))$. The forgetful functor $\fr: Z(\vect_G^\omega) \to \vect_G^\omega$ determines a fusion subcategory $\rep(G) \hookrightarrow Z(\vect_G^\omega)$, which is nothing but the preimage of the direct sums of the tensor unit $\one_{\vect_G^\omega}$. The Lagrangian algebra $A=\fr^\vee(\one_{\vect_G^\omega})$ is nothing but
the algebra $\fun(G)$ of functions on $G$ in $\rep(G)$. We have
\bnu
\item the condensable subalgebras of $A$ are given by $\fun(G/H)$ for subgroups $H\subset G$, 
\item the fusion subcategories of $\vect_G^\omega$ are $\vect_H^{\omega|_H}$,
\enu
and $\varphi(\fun(G/H))=\vect_H^{\omega|_H}$. 
\end{prop}

\subsection{Symmetric fusion categories} \label{sec:sfc}

A braided fusion category $\EC$ is called {\it a symmetric fusion category} (SFC) if $\EC'=\EC$. Throughout this work, we use $\EE$ to denote a SFC. 

\smallskip
Let $G$ be a finite group. The category of representations of $G$, denoted by $\Rep(G)$, is a SFC. Such SFC's $\EE$ are characterized by the fact that there is a braided monoidal functor $F:\EE\to \vect$ (unique up to isomorphisms), also called a symmetric fiber functor. Moreover, we have $G\simeq \aut(F)$ as groups iff $\EE \simeq \rep(G)$ as braided fusion categories. In this case, $F$ is just the usual forgetful functor $\rep(G) \to \vect$. 

%When $\EE=\rep(G)$, there is another convenient way to characterize the symmetric fiber functor $F: \EE \to \vect$. Let $I: \vect \to \EE$ be the right adjoint functor of $F$ and $A=I(\one)$ the commutative algebra in $\EE$. The monoidal category $\EE_A$ of right $A$-modules in $\EE$ is monoidally equivalent to $\vect$. Moreover, the functor $F$ can be identified with the free induction functor $-\otimes A: \EE \to \EE_A=\vect$. If $\EE=\rep(G)$, this algebra $A=I(\one)$ is nothing but the algebra $\fun(G)$ of $\Cb$-valued functions on $G$. 

In this work, we are interested in condensable algebras in $\rep(G)$ for a finite group $G$. The following classification result can be found in \cite{ko}. 

\begin{thm} \label{thm:KO}
Let $\fun(G/H)$ be the algebra of the functions on the coset $G/H$ for a subgroup $H$ (or equivalently, the functions on $G$ that are invariant under the action of $H$ by translations).
\bnu
\item $\fun(G/H)$ is a condensable algebra in $\rep(G)$. 
\item If $A$ is a condensable algebra in $\rep(G)$, then there is a subgroup $H$ such that $A\simeq \fun(G/H)$. Moreover, there is a canonical symmetric monoidal equivalence $\rep(G)_A\simeq \rep(H)$. 
\item The forgetful functor $\rep(G) \to \rep(H)$ (forgetting $g$-actions for $g\notin H$) and the induction functor $\mathrm{Ind}_H^G: \rep(H) \to \rep(G)$ are left and right adjoints of each other. Using $A\simeq \fun(G/H)$ and $\rep(G)_A\simeq \rep(H)$, these two functors can be identified with the functor $-\otimes A: \rep(G) \to \rep(G)_A$ and $\rep(G)_A \xrightarrow{forget} \rep(G)$, respectively. 
\enu
\end{thm}

\begin{rema} \label{rema:sym-breaking}{\rm
When the SFC $\rep(G)$ is viewed as the symmetry of a bosonic SPT order, condensing the algebra $\fun(G/H)$ for a subgroup $H\subset G$ should be viewed as breaking the symmetry $G$ to $H$, or equivalently, breaking $\rep(G)$ to $\rep(H)$. 
}
\end{rema}

Note that $\Cb=\fun(G/G)$ is the trivial $\rep(G)$-algebra $\one_{\rep(G)}$. Let $A=\fun(G)$. We have $\rep(G)_A=\rep(G)_A^0\simeq\vect$. Moreover, the free induction functor $-\otimes A: \rep(G) \to \rep(G)_A$ can be identified with the forgetful functor $\fr: \rep(G) \to \vect$. Moreover, the algebra $A=I(\one)$, where $I$ is the right adjoint functor of the forgetful functor $\fr$, is maximal in the sense that $\fun(G/H)$ is a subalgebra of $\fun(G)$ for any subgroup $H$ of $G$.

\smallskip
More generally, according to Deligne \cite{deligne}, a SFC is braided equivalent to $\rep(G,z)$, where $G$ is a finite group, and $z\in G$ is a central element such that $z^2=1$ (see also \cite{egno}). The SFC $\rep(G,z)$ is the same as $\rep(G)$ as fusion categories, but is equipped with a new braiding. More precisely, for $X,Y\in \rep(G,z)$, the new braiding $c_{X,Y}^z: X\otimes Y \to Y\otimes X$ is defined as follows: 
$$
c^z_{X,Y}(x\otimes y) = (-1)^{mn} y\otimes x, 
$$
for all $x\in X, y\in Y$ such that $zx=(-1)^mx, zy=(-1)^ny$, where $m$ and $n$ are either $0$ or $1$.

When $G=\Zb_2$, the SFC $\rep(\Zb_2,z)$ is nothing but the category $\svect$ of super-vector spaces (or $\Zb_2$-graded vector spaces) with $\Zb_2$-graded symmetric braidings.

For any SFC $\EE$, there is always a braided monoidal functor $F: \EE \to \svect$ (unique up to isomorphisms), also called super fiber functor. Let $G=\aut(F)$ be the monoidal natural automorphisms of $F$ and $z\in G$ be the parity automorphism of $F$ (i.e. $z_x: F(x) \to F(x)$ is the parity automorphism on $F(x)$ for $x\in \EE$). Then we have $\EE \simeq \rep(G,z)$. 

\medskip
Let $A$ be a condensable algebra in $\EE$. The category $\EE_A=\EE_A^0$ is automatically a SFC. The free induction functor $-\otimes A: \EE \to \EE_A$ is symmetric monoidal and should be viewed as a symmetry-breaking process (or a condensation). We have the following Lemma as a corollary of  Prop.\,\ref{prop:embed}. 

\begin{lem}
Let $\EF$ be fusion subcategory of $\EE$ and $A$ is a condensable algebra of $\EE$ such that $A\cap \EF=\one$. We have a braided full embedding $-\otimes A: \EF \rightarrow \EE_A$. 
\end{lem}

In other words, the symmetry $\EF$ is preserved under the symmetry-breaking process $-\otimes A: \EE \to \EE_A$.

\subsection{Braided fusion categories over $\EE$}  \label{sec:bfc-over-e}

In this subsection, we recall the notion of a braided fusion category over a SFC $\EE$ and its basic properties from \cite{dno}. 

\begin{defn} {\rm
A braided fusion category over $\EE$, or a \bfce, is a braided fusion category $\EC$ equipped with a braided full embedding $\eta_\EC: \EE \hookrightarrow \EC'$. A fusion $/\EE$-subcategory of $\EC$ is a fusion subcategory containing $\EE$. A \bfce ~$\EC$ is called {\it non-degenerate}, if, in addition, $\EC'=\EE$. 
}
\end{defn}

We will abbreviate a non-degenerate \bfce ~to a \nbfce.

\begin{defn} {\rm
A braided $/\EE$-functor $f: \EC \to \ED$ between two \bfce's $\EC$ and $\ED$ is a braided functor preserving the embeddings of $\EE$, i.e. $\eta_\ED\simeq f\circ \eta_\EC$. 
}
\end{defn}

\begin{rema} \label{rema:aut-C}  {\rm
We denote the category of braided $/\EE$-autoequivalences of $\EC$ by $\mathrm{Aut}(\EC)$ and the underlining group by $\underline{\mathrm{Aut}}(\EC)$. Note that $\underline{\mathrm{Aut}}(\EC)$ is the trivial group when $\EC=\EE$. 
}
\end{rema}

Let $\EC$ be a braided fusion category. Let $R: \EC \to \EC\boxtimes \overline{\EC}$ be the right adjoint of the tensor product functor $\EC\boxtimes \overline{\EC} \xrightarrow{\otimes} \EC$, which factors through the canonical braided functor $\EC \boxtimes \overline{\EC} \to Z(\EC)$. Set $L_\EC:=R(\one_\EC)$. It is a condensable algebra in $\EC\boxtimes \overline{\EC}$ and decomposes as $L_\EC=\oplus_{i\in O(\EC)} i\boxtimes i^\ast$. Note that $\fpdim(L_\EC)=\fpdim(\EC)$. Similarly, we have the condensable algebra $L_\EE=R(\one_\EE)=\oplus_{i\in O(\EE)} i\boxtimes i^\ast$ induced from $\otimes: \EE \boxtimes \EE \to \EE$ and its right adjoint functor $R$. If $\EC$ is a \bfce, it is clear that $L_\EE$ is a condensable subalgebra of $L_\EC$. The condensation of $L_\EE$ break the symmetry from $\EE \boxtimes \EE$ to $\EE$. 

\begin{rema} \label{rema:E-alg} {\rm
When $\EE=\rep(G)$, we have $\rep(G)\boxtimes \rep(G)=\rep(G\times G)$ and the tensor product functor $\otimes: \EE \boxtimes \EE \to \EE$ can be identified with the forgetful functor $\rep(G\times G) \to \rep(G)$, where $G$ is the subgroup of $G\times G$ via diagonal map $\Delta: G \hookrightarrow G\times G$. By Thm.\,\ref{thm:KO}, the algebra $L_{\rep(G)}$ can be identified with the algebra $\fun(G\times G/G)$ of functions on the coset $G\times G/G$. According to Remark\,\ref{rema:sym-breaking}, condensing the algebra $L_\EE$ amounts to breaking the symmetry $G \times G$ to $G$. 
}
\end{rema}

%When $\EC=\EE$ is symmetric, we have $\EE \simeq(\EE\boxtimes\EE)_{L_\EE}=(\EE\boxtimes\EE)_{L_\EE}^0$. 

Let $\EC$ and $\ED$ be two \bfce's. The relative tensor product $\EC\boxtimes_\EE\ED$ is well-defined as a category \cite{T,eno2009} and as a fusion category \cite{dno,kz}, which can be identified with the fusion category $(\EC\boxtimes \ED)_{L_\EE}$ \cite{dno}. Moreover, in this case, we have $(\EC\boxtimes \ED)_{L_\EE} = (\EC\boxtimes \ED)_{L_\EE}^0$. Therefore, $\EC\boxtimes_\EE\ED$ has a canonical structure of a braided fusion category. Since $L_\EE \cap (\EE\boxtimes \one)=\one\boxtimes \one$, by Prop.\,\ref{prop:embed}, we obtain a braided full embedding $-\otimes L_\EE: \EE \hookrightarrow (\EC\boxtimes_\EE \ED)'$. Therefore, $\EC\boxtimes_\EE \ED$ is a \bfce ~\cite{dno}. The relative tensor product $\boxtimes_\EE$ amounts to break the symmetry $\EE\boxtimes \EE$ on $\EC\boxtimes \ED$ to $\EE$. It forms an associative tensor product in the 2-category of \bfce's. 

%For general symmetry breaking process, we have the following result.  
\begin{prop} {\rm \cite[Cor.\,4.6]{dno}}
Let $A$ be a condensable algebra in a \bfce ~$\EC$ and $A\cap \EE=\one$. Then $\EC_A^0$ is a \bfce. If, in addition, $\EC$ is an \nbfce, so is $\EC_A^0$. 
\end{prop}

%Apply this result to the case of $\EC\boxtimes \ED$ for two \nbfce's $\EC$ and $\ED$ and $A=L_\EE$, we obtain the following corollary. \begin{cor} %{\rm \cite{dno}}If both $\EC$ and $\ED$ are \nbfce's, so is $\EC\boxtimes_\EE\ED$. \end{cor}

% and that of \nbfce's. 

%The symmetry breaking process can be formulated even more generally. Let $\EH$ be another symmetric fusion category and $f: \EE \to \EH$ be a braided monoidal functor. Let $I$ be the right adjoint functor of $\otimes \circ (f\boxtimes \id_\EH): \EE\boxtimes \EH \to \EH\boxtimes \EH \to \EH$ and $A=I(\one_\EH)$. We have $\EE\boxtimes_\EE \EH \simeq (\EE\boxtimes \EH)_A\simeq \EH$. Then, for a \bfce ~$\EC$, the category $\EC\boxtimes_\EE \EH:=(\EC\boxtimes \EH)_A^0$ is a $\mathrm{BFC}_{/\EH}$. In other words, $-\boxtimes_\EE\EH$ defines a base-changing 2-functor from the 2-category $\mathbf{BFC}_{/\EE}$ of \bfce's to the category $\mathbf{BFC}_{/\EH}$ of $\mathrm{BFC}_{/\EH}$'s. % and a base-changing functor $\mathbf{UMTC}_{/\EE} \to \mathbf{UMTC}_{/\EH}$. 

\subsection{Unitary braided fusion categories} \label{sec:unitary}

\begin{defn} \label{def:star-cat} {\rm
A $\ast$-category $\EC$ is a $\Cb$-linear category equipped with a functor $\ast: \EC \to \EC^\op$ 
which acts trivially on the objects and is anti-linear and involutive on morphisms, i.e. $\ast:\hom(A,B) \to \hom(B,A)$ is defined so that 
\be \label{eq:dagger}
(g \circ f)^\ast = f^\ast \circ g^\ast, \quad\quad (\lambda f)^\ast = \bar{\lambda} f^\ast,\quad\quad f^{\ast\ast} = f. 
\ee
for $f: U \to V$, $g: V \to W$, $h: X \to Y$, $\lambda \in \Cb^\times$.
A $\ast$-category is called {\it unitary} if $\ast$ satisfies the positive condition: $f\circ f^\ast =0$ implies $f=0$. 
}
\end{defn}

\begin{defn}  \label{def:unitary} {\rm
A unitary fusion category $\EC$ is a fusion category and a unitary category such that $\ast$ is compatible with the monoidal structures, i.e.
\bea
&(g\otimes h)^\ast=g^\ast \otimes h^\ast,\quad\quad \forall g: V\to W, h: X \to Y, & \\
&\alpha_{X,Y,Z}^\ast=\alpha_{X,Y,Z}^{-1},\quad l_X^\ast=l_X^{-1},\quad r_X^\ast=r_X^{-1}. & \label{eq:unitary-asso-unit}
\eea
A unitary braided fusion category is a unitary fusion category and is braided so that
$c_{X,Y}^\ast = c_{X,Y}^{-1}$ for all $X,Y$. 
}
\end{defn}

We abbreviate a unitary fusion category to a UFC, and a unitary braided fusion category to a UBFC. In a UFC, the hom spaces are actually finite dimensional Hilbert spaces, and the left duals coincide with the right duals, i.e. ${}^\ast x = x^\ast$ for $x\in \EC$. A fusion subcategory of a UFC is automatically a UFC.

\begin{rema} {\rm
A convenient way to check the unitarity of a given fusion category is to check if one can find a basis of the hom spaces such that fusion matrices in this basis are all unitary. This is enough to promote a fusion category to a unitary fusion category \cite{yamagami,galindo}. 
}
\end{rema}

Let $\EC$ be a UFC. We would like to know if the Drinfeld center $Z(\EC)$ is unitary. Let $Z^\ast(\EC)$ be the unitary center that is defined as the subcategory of $Z(\EC)$ such that the half-braidings in $Z^\ast(\EC)$ are all unitary. 
\begin{prop} {\rm \cite{galindo,ghr}} \label{prop:unitary}
Every braiding of a $\mathrm{UFC}$ is unitary. In particular, for a $\mathrm{UFC}$ $\EC$, we have $Z^\ast(\EC)=Z(\EC)$ and $Z(\EC)$ is a \ubfc.   
\end{prop}

More generally, if $\EC$ is a UFC, the natural replacement of a $\EC$-module category is a $\EC$-module $\ast$-category. A unitary functor is a functor preserving the $\ast$-structure. 

\begin{thm} {\rm \cite{ghr}}
The monoidal category $\Fun_\EC^\ast(\EM, \EM)$ of $\EC$-module $\ast$-functors is a unitary fusion category that is monoidally equivalent to $\Fun_\EC(\EM,\EM)$. 
\end{thm}

It is well-known that a UFC has a unique spherical structure \cite{kitaev}, and the Frobenius-Perron dimensions coincide with the quantum dimensions \cite{eno2002}. A non-degenerate \ubfc ~is automatically equipped with the structure of a unitary modular tensor category (UMTC) (introduced first in \cite{MS}). We will not define a UMTC explicitly (see for example \cite{turaev}). For the purpose of this work, it is enough to take the non-degenerate \ubfc ~as the definition of a UMTC. 

\begin{expl} \label{expl:unitary} {\rm
We give some examples of unitary (braided) fusion categories.
\bnu
\item Let $G$ be a finite group. The fusion category $\rep(G)$ has a canonical structure of UFC. As a consequence, symmetric fusion categories are all unitary. By Prop.\,\ref{prop:unitary}, the Drinfeld center $Z(\EE)$ of a SFC $\EE$ is also unitary. 
\item Since $H^n(G,U(1))=H^n(G,\Cb^\times)$ by the universal coefficient theorem, every pointed fusion category is a unitary fusion category \cite{ghr}, i.e. $\vect_G^\omega$ is a UFC for $\omega\in H^3(G,U(1))$. The Drinfeld center $Z(\vect_G^\omega)$ is a UMTC.
\enu
}
\end{expl}

A lot of constructions for (non-degenerate) braided fusion categories work automatically in the unitary cases. For example, given a UMTC $\EC$ and a condensable algebra $A$ in $\EC$, it is easy to check that $\EC_A$ is a UFC and $\EC_A^0$ is a UMTC \cite{bek}. Most of the results in this work holds for both unitary and non-unitary cases. We will mention explicitly when results in unitary and non-unitary cases are different.

\begin{defn} \label{def:umtce}
{\rm
A unitary braided fusion category over $\EE$, or a \ubfce, is a UBFC $\EC$ equipped with a braided full embedding $\eta_\EC: \EE \hookrightarrow \EC'$. A fusion $/\EE$-subcategory of $\EC$ is a fusion subcategory containing $\EE$. A \ubfce ~$\EC$ is called {\it non-degenerate}, or a unitary modular tensor category over $\EE$ (or \umtce), if, in addition, $\EC'=\EE$. 
}
\end{defn}

We will abbreviate a non-degenerate \ubfce ~(or a unitary modular tensor category over $\EE$) ~to a \umtce.

\begin{defn} {\rm
A braided $/\EE$-functor $f: \EC \to \ED$ between two \ubfce's $\EC$ and $\ED$ is a braided unitary functor preserving the embeddings of $\EE$, i.e. $\eta_\ED\simeq f\circ \eta_\EC$. 
}
\end{defn}

We use the same notations as in Remark\,\ref{rema:aut-C} for the unitary cases.

\section{The group $\EM_{ext}(\EE)$ of the modular extensions of $\EE$}  \label{sec:mext}

\subsection{Modular extensions of a \mtce} \label{sec:mext-def}

\begin{defn} {\rm
A UMTC containing a UBFC $\EC$ is a pair $(\EM, \iota_\EM)$, where $\EM$ is a UMTC and $\iota_\EM: \EC \hookrightarrow \EM$ is a braided full embedding. 
}
\end{defn}

\begin{rema} {\rm
If we drop the assumption on the unitarity on both $\EM$ and $\EC$, we obtain the notion of a {\it non-degenerate extension} $(\EM, \iota_\EM)$ of a \bfce ~$\EC$. It should be interesting mathematically. We will discuss a little bit about it in Sec.\,\ref{sec:witt}.
}
\end{rema}

\begin{rema} {\rm
If $(\EM, \iota_\EM)$ is a UMTC $\EM$ containing a UBFC $\EC$, then $(\overline{\EM}, \overline{\iota_\EM}:=\iota_\EM: \overline{\EC} \hookrightarrow \overline{\EM})$ is automatically a UMTC containing $\overline{\EC}$. 
}
\end{rema}

\begin{defn} \label{def:mext}
{\rm
Let $\EC$ be a \umtce. A {\it modular extension} of $\EC$ is a UMTC containing $\EC$, i.e. $(\EM, \iota_\EM)$, such that $\EE'|_\EM=\EC$. 
}
\end{defn}

\begin{rema} \label{rema:minimal} {\rm
The notion of a modular extension of $\EC$ was first introduced by M\"{u}ger in \cite[Conjecture\,5.2]{mueger1}, and was called ``a minimal modular extension of $\EC$''. In this work, we only study the minimal modular extensions. Importantly, it is not yet clear to us what the non-minimal ones (dropping the condition $\EE'|_\EM=\EC$) represent in physics (see also Remark\,\ref{rema:non-minimal}). Therefore, we drop the term ``minimal'' until the last section, where we raise some questions on non-minimal modular extensions. 
}
\end{rema}

\begin{rema} {\rm
For a given \umtce ~$\EC$, it is possible that there is no modular extension of $\EC$. An example was constructed by Drinfeld \cite{drinfeld}. %It was believe that such a UBFC over $\EE$ describes a 2+1D topological order with an anomalous (not on-site) symmetry $\EE$. 
}\end{rema}

\begin{expl} \label{expl:E-ZE}  {\rm
Let $\iota_0: \EE \hookrightarrow Z(\EE)$ be the canonical embedding induced by the central functor $\id_\EE: \EE \to \EE$. Then the pair $(Z(\EE), \iota_0)$ gives a modular extension of $\EE$. 
}
\end{expl}

\begin{expl} {\rm
Let $\EM$ be a UMTC containing $\EE$. Then $\EE'|_{\EM}$ is a \umtce
~\cite{mueger2} and $\EM$ is a modular extension of $\EE'|_{\EM}$.  
}
\end{expl}

\begin{defn} \label{def:eq-mext}  {\rm
Two modular extensions $(\EM, \iota_\EM)$ and $(\EN, \iota_\EN)$ of a \mtce ~$\EC$ are equivalent if there is a unitary braided monoidal equivalence $f: \EM \xrightarrow{\simeq} \EN$ such that $f\circ  \iota_\EM \simeq \iota_\EN$. 
}
\end{defn}

We denote the set of the equivalence classes of the modular extensions of a \mtce ~$\EC$ by $\EM_{ext}(\EC)$. %that of the equivalence classes of the modular extension of a \umtce ~$\EC$ by $\EM_{ext}^u(\EC)$. Sometimes, these two sets equal (see Sec.\,\ref{sec:mext-repG}) but are not equal in general (see Sec.\,\ref{sec:mext-svect}). 
For simplicity, we also abbreviate the pair $(\EM, \iota_\EM)$ simply to $\EM$. Also note that we have the identity $\fpdim(\EM)=\fpdim(\EC)\fpdim(\EE)$. Since the Frobenius-Perron dimension of the modular extension of $\EC$ is fixed, according to \cite{bnrw}, the set $\EM_{ext}(\EC)$, if not empty, must be finite. 

\begin{rema} \label{rema:aut-C-2}  {\rm
If $\EM_{ext}(\EC)$ is not empty, there is a natural action of 
$\mathrm{Aut}(\EC)$ (recall Remark\,\ref{rema:aut-C}) on the category of the modular extensions of $\EC$ defined by $\phi \cdot (\EM, \iota_M) := (\EM, \iota_\EM \circ \phi)$ for $\phi \in \mathrm{Aut}(\EC)$. 
This action descends to a natural $\underline{\mathrm{Aut}}(\EC)$-action on $\EM_{ext}(\EC)$. 
}
\end{rema}

\begin{lem} \label{lem:MEN}
If $\EM_{ext}(\EC)$ and $\EM_{ext}(\ED)$ are not empty, then $\EM_{ext}(\EC\boxtimes_\EE \ED)$ is not empty, and there is a well-defined map 
$$
\boxtimes_\EE^{(-,-)}: \EM_{ext}(\EC) \times \EM_{ext}(\ED) \to \EM_{ext}(\EC\boxtimes_\EE \ED).
$$
More explicitly, let $(\EM, \iota_\EM:\EC \hookrightarrow \EM)$ and $(\EN, \iota_\EN:\ED\hookrightarrow \EN)$ be the modular extensions of two \mtce's $\EC$ and $\ED$, respectively. Then $\EC\boxtimes_\EE \ED$ is a \mtce ~and the pair 
$$
\EM \boxtimes_\EE^{(\iota_\EM,\iota_\EN)} \EN:= \left( (\EM \boxtimes\EN)_{L_\EE}^0\, , \quad \iota_\EM \boxtimes_\EE \iota_\EN: (\EC\boxtimes\ED)^0_{L_\EE} \hookrightarrow (\EM \boxtimes\EN)_{L_\EE}^0 \right)
$$ 
is a modular extension of $\EC\boxtimes_\EE\ED$. 
\end{lem}
\begin{proof}
Notice that we have the following embeddings of braided fusion categories:
$$
\EE = \EE \boxtimes_\EE \EE = (\EE\boxtimes \EE)_{L_\EE}^0 \hookrightarrow (\EC \boxtimes \ED)_{L_\EE}^0 \hookrightarrow (\EM \boxtimes \EN)_{L_\EE}^0\, .
$$
It is clear that $\EC\boxtimes_\EE \ED$ is contained in $\EE'|_{(\EM \boxtimes \EN)_{L_\EE}^0}$ and $\fpdim(\EE)\fpdim(\EC\boxtimes_\EE \ED) =\fpdim(\EM \boxtimes \EN)_{L_\EE}^0$. Therefore, we must have $\EC\boxtimes_\EE \ED=\EE'|_{(\EM \boxtimes \EN)_{L_\EE}^0}$. This implies both that $\EC\boxtimes_\EE \ED$ is a \mtce ~and that $\left( (\EM \boxtimes\EN)_{L_\EE}^0,\iota_\EM \boxtimes_\EE \iota_\EN \right)$ is a modular extension of $\EC\boxtimes_\EE \ED$.  
\end{proof}

\begin{rema} {\rm
We use the notation $\EM \boxtimes_\EE^{(\iota_\EM,\iota_\EN)} \EN$ to suggest that it 
can indeed be viewed as some kind of relative product of two modular extensions. The superscript in $\boxtimes_\EE^{(\iota_\EM,\iota_\EN)}$ reminds readers that it is not the relative tensor product $\boxtimes_\EE$ in the usual sense. 
}
\end{rema}

\begin{prop} \label{prop:comm}
The tensor product $\boxtimes_\EE^{(-,-)}$ is commutative, i.e. $\EM \boxtimes_\EE^{(\iota_\EM,\iota_\EN)} \EN\simeq \EN \boxtimes_\EE^{(\iota_\EN,\iota_\EM)} \EM$. 
\end{prop}
\begin{proof}
It is enough to check that the functor $x\boxtimes y \mapsto y\boxtimes x$ from $\EM \boxtimes \EN$ to $\EN\boxtimes \EM$ carries $L_\EE$ to $L_\EE$. This follows from the fact that the tensor product functor $\otimes: \EE \boxtimes \EE \to \EE$ is isomorphic to the reversed tensor product functor $\otimes^{\rev}$ due to the symmetric braidings. 
\end{proof}

\begin{prop} \label{prop:associativity}
Let $(\EL, \iota_\EL:\EB \hookrightarrow \EL)$, $(\EM, \iota_\EM:\EC \hookrightarrow \EM)$ and $(\EN, \iota_\EN:\ED\hookrightarrow \EN)$ be the modular extensions of three \mtce's $\EB$, $\EC$ and $\ED$, respectively. There is a canonical associator  
\be \label{eq:alpha}
\EL \boxtimes_\EE^{(\iota_\EL, \iota_\EM\boxtimes_\EE \iota_\EN)} (\EM \boxtimes_\EE^{(\iota_\EM, \iota_\EN)} \EN) \xrightarrow{\simeq} 
(\EL \boxtimes_\EE^{(\iota_\EL, \iota_\EM)} \EM) \boxtimes_\EE^{(\iota_\EL \boxtimes_\EE \iota_\EM, \iota_\EN)} \EN.
\ee
\end{prop}
\begin{proof}
%By \cite[Lemma\,4.13]{ffrs}, the condensable algebra $L_\EE$ in $\EE \boxtimes \EE \hookrightarrow \EL \boxtimes (\EM \boxtimes \EN)_{L_\EE}^0$ corresponds to a unique (up to isomorphisms) condensable algebra $A$ over the condensable algebra $\one_\EL\boxtimes L_\EE$ in $\EL\boxtimes \EM \boxtimes \EN$ such that $( \EL \boxtimes (\EM \boxtimes \EN)_{L_\EE})_{L_\EE}= ( \EL \boxtimes \EM \boxtimes \EN )_A$. It is easy to check that $A=(L_\EE \boxtimes \one) \otimes (\one\boxtimes L_\EE)$. The algebra $A$ is canonically isomorphic to $B= (\one\boxtimes L_\EE) \otimes (L_\EE \boxtimes \one)$ via the obvious braiding, which is an algebra homomorphism because the braiding is symmetric. 
Let $R_1$ and $R_2$ be the right adjoint functors of the following two central functors $\otimes\circ (\id_\EE \boxtimes \otimes), \otimes\circ (\otimes \boxtimes \id_\EE): \EE\boxtimes \EE \boxtimes \EE \to \EE$, respectively. Clearly, $R_1\simeq R_2$ by the associativity of $\otimes$. Then we obtain the following braided monoidal equivalences:
$$
\alpha: \,  \left( \EL \boxtimes (\EM \boxtimes \EN)_{L_\EE}^0 \right)_{L_\EE}^0 \simeq 
\left( \EL \boxtimes \EM \boxtimes \EN \right)_{R_1(\one_\EE)}^0 \simeq \left( \EL \boxtimes \EM \boxtimes \EN \right)_{R_2(\one_\EE)}^0 \simeq
 \left( (\EL \boxtimes \EM)_{L_\EE}^0 \boxtimes \EN \right)_{L_\EE}^0. 
$$
We can show that $\alpha \circ (\iota_\EL \boxtimes_\EE (\iota_\EM \boxtimes_\EE \iota_\EN))=(\iota_\EL \boxtimes_\EE \iota_\EM) \boxtimes_\EE \iota_\EN$ by computing the image of the object $b\boxtimes_\EE (c\boxtimes_\EE d)$ in $\EB\boxtimes_\EE(\EC\boxtimes_\EE \ED)$. It further implies that $\alpha$ also gives the associator in Eq.(\ref{eq:alpha}). In particular, we obtain $\EB\boxtimes_\EE(\EC\boxtimes_\EE\ED) \simeq (\EB\boxtimes_\EE \EC) \boxtimes_\EE \ED$ as \mtce's.
\end{proof}

Prop.\,\ref{prop:associativity} implies that following diagram
$$
\xymatrix{
\EM_{ext}(\EB) \times \EM_{ext}(\EC) \times \EM_{ext}(\ED)  \ar[rrr]^{\id_{\EM_{ext}(\EB)} \,\, \times \,\, \boxtimes_\EE^{(-,-)}} \ar[d]_{\boxtimes_\EE^{(-,-)} \,\, \times \,\, \id_{\EM_{ext}(\ED)}} & & & \EM_{ext}(\EB) \times \EM_{ext}(\EC\boxtimes_\EE \ED) \ar[d]^{\boxtimes_\EE^{(-,-)}} \\
\EM_{ext}(\EB\boxtimes_\EE \EC) \times \EM_{ext}(\ED)  \ar[rrr]^{\boxtimes_\EE^{(-,-)}} & & & \EM_{ext}(\EB\boxtimes_\EE \EC \boxtimes_\EE \ED)\, .
}
$$
is commutative when all three sets $\EM_{ext}(\EB), \EM_{ext}(\EC), \EM_{ext}(\EB)$ are not empty. 

\subsection{The finite abelian group structure on $\EM_{ext}(\EE)$}  \label{sec:mext-E-group}

The set $\EM_{ext}(\EE)$ is not empty because $(Z(\EE), \iota_0)\in \EM_{ext}(\EE)$. The product $\boxtimes_\EE^{(-, -)}: \EM_{ext}(\EE) \times \EM_{ext}(\EE) \to \EM_{ext}(\EE)$ defines a binary multiplication on the set $\EM_{ext}(\EE)$. In this subsection, we would like to show that the set $\EM_{ext}(\EE)$, together with the binary multiplication $\boxtimes_\EE^{(-, -)}$ and the identity element $(Z(\EE),\iota_0)$, is a finite abelian group. 

\medskip
By Prop.\,\ref{prop:associativity} and Prop.\,\ref{prop:comm}, this multiplication $\boxtimes_\EE^{(-, -)}$ is also associative and commutative. It remains to show the existence of the inverses and the identity element.  

\begin{lem} \label{lem:cclc}
Let $\EM$ and $\EN$ be braided fusion categories equipped with braided embeddings $\EC \hookrightarrow \EM$ and $\overline{\EC}\hookrightarrow\EN$. 
\bnu
\item The functor $-\otimes L_\EC: \EM \to (\EM\boxtimes \EN)_{L_\EC}$ defined by $x \mapsto (x\boxtimes \one_\EN) \otimes L_\EC$ is fully faithful and monoidal, and its restriction to $\EC'|_\EM$ gives a braided full embedding $\EC'|_\EM \hookrightarrow (\EM\boxtimes \EN)_{L_\EC}^0$. 
\item The functor $-\otimes L_\EC$ is a monoidal equivalence iff $\EN=\overline{\EC}$. In this case, its restriction to $\EC'|_\EM$ is a braided monoidal equivalence, i.e. $\EC'|_\EM\simeq (\EM\boxtimes \overline{\EC})_{L_\EC}^0$ as braided fusion categories. 
\enu
\end{lem}
\begin{proof}
Part 1 is a special case of Lemma\,\ref{lem:embed} and Prop.\,\ref{prop:embed} because $L_\EC \cap \EM=\one_\EM$. 
%The functor $-\otimes L_\EC$ is fully faithful because $\hom_{L_\EC}((x\boxtimes \one_\EN) \otimes L_\EC, (y\boxtimes \one_\EN) \otimes L_\EC)) \simeq \hom_{\EM\boxtimes \EN}(x\boxtimes \one_\EN, (y\boxtimes \one_\EN) \otimes L_\EC)) = \hom_\EM(x,y)$. It is clearly monoidal. It is easy to see that $(x\boxtimes \one_\EN) \otimes L_\EC$ is a local $L_\EC$-module iff $x \in \EC'|_\EM$. Therefore, we obtain a monoidal embedding $\EC'|_\EM \hookrightarrow (\EM\boxtimes \EN)_{L_\EC}^0$, which also respects the braidings. 
For Part 2, note that $\fpdim((\EM\boxtimes \EN)_{L_\EC})=\fpdim(\EM)\fpdim(\EN)/\fpdim(L_\EC)\geq \fpdim(\EM)$. By \cite[Prop.\,2.19]{eo}, the functor $-\otimes L_\EC$ is a monoidal equivalence iff $\EN=\overline{\EC}$, i.e.  $\EM\simeq (\EM\boxtimes \overline{\EC})_{L_\EC}$ as UFC's.  
%That $(x\boxtimes \one) \otimes L_\EC$ is a local $L_\EC$-module iff $x \in \EC'|_\EM$ implies 
Similarly, by checking Frobenius-Perron dimensions, we obtain $\EC'|_\EM \simeq (\EM\boxtimes \overline{\EC})_{L_\EC}^0$ as UBFC's. 
\end{proof}

\begin{lem} \label{lem:cze}
 Let $\EC$ be a \mtce ~and $(\EM,\iota_\EM)$ a modular extension of $\EC$. We have 
$$
\EM \boxtimes_\EC^{(\iota_\EM, \overline{\iota_\EM})} \overline{\EM}\,\, := \,\, \left( (\EM\boxtimes\overline{\EM})_{L_\EC}^0, \,\, (\EC \boxtimes \overline{\EC})_{L_\EC}^0 \hookrightarrow (\EM\boxtimes\overline{\EM})_{L_\EC}^0 \right) \simeq (Z(\EE), \iota_0).
$$
\end{lem}
\begin{proof}
Consider the Lagrangian algebra $L_\EM$ in $\EM\boxtimes \overline{\EM}$. By Lemma\,\ref{lem:cclc}, $\EM=(\EM\boxtimes\overline{\EM})_{L_\EM}$ (via the functor $x\mapsto (x\boxtimes \one)\otimes L_\EM$). Because the functor $\EE \hookrightarrow (\EM\boxtimes\overline{\EM})_{L_\EC}^0 \xrightarrow{-\otimes_{L_\EC}L_\EM} \EM$ coincides with the embedding $\EE \hookrightarrow \EM$, the image of $(\EM\boxtimes\overline{\EM})_{L_\EC}^0$ under the functor $-\otimes_{L_\EC}L_\EM$, denoted by $\EB$, contains $\EE$. Note that the functor $-\otimes_{L_\EC}L_\EM$ is a central functor and its right adjoint is the forgetful functor. By Thm.\,\ref{thm:dmno-1}, $\EB$ is monoidally equivalent to the fusion category $( (\EM\boxtimes\overline{\EM})_{L_\EC}^0)_{L_\EM}$. It is easy to check that $\fpdim( ( (\EM\boxtimes\overline{\EM})_{L_\EC}^0)_{L_\EM} ) = \fpdim(\EE)$. Therefore, $\EB=\EE$. By Theorem\,\ref{thm:dmno}, we must have $(\EM\boxtimes\overline{\EM})_{L_\EC}^0\simeq Z(\EE)$, and the functor $-\otimes_{L_\EC} L_\EM$ coincides with the forgetful functor $Z(\EE) \to \EE$. This implies that the composed functor $\EE \xrightarrow[\simeq]{-\otimes L_\EC} (\EC \boxtimes \overline{\EC})_{L_\EC}^0 \hookrightarrow (\EM\boxtimes\overline{\EM})_{L_\EC}^0\simeq Z(\EE)$ coincides with $\iota_0$. 
\end{proof}

\begin{cor} \label{cor:MM=1}
Let $(\EM, \iota_\EM)$ be a modular extension of $\EE$ and $\overline{\iota_\EM}$ the same functor $\EE=\overline{\EE} \hookrightarrow \overline{\EM}$. The pair $(\overline{\EM}, \overline{\iota_\EM})$ is also a modular extension of $\EE$. We have 
\be \label{eq:MM=1}
\EM \boxtimes_\EE^{(\iota_\EM, \overline{\iota_\EM})} \overline{\EM} \simeq (Z(\EE), \iota_0). 
\ee
\end{cor}

By Prop.\,\ref{prop:comm} and Cor.\,\ref{cor:MM=1}, we obtain $(Z(\EE), \iota_0) \simeq (\overline{Z(\EE)}, \overline{\iota_0})$. More directly, there is a braided equivalence $\overline{Z(\EE)}\simeq Z(\EE^\rev) = Z(\EE)$ defined by $(x, z_{x,-}) \mapsto (x, z_{x,-}^{-1})$ such that it is compatible with $\iota_0$ and $\overline{\iota_0}$.  It remains to show that $(Z(\EE), \iota_0)$ is the identity element. 

\begin{lem} \label{lem:unit-of-group}
Let $(\EM, \iota_\EM: \EC \hookrightarrow \EM)$ be a \umtc ~containing $\EC$. If $\EM$ also contains $\EE$ and $\EC\subset \EE'|_\EM$, then there is a canonical braided equivalence $g: \EM \xrightarrow{\simeq} (\EM\boxtimes Z(\EE))_{L_\EE}^0$ such that $g\circ \iota_\EM \simeq (-\boxtimes \one_{Z(\EE)}) \otimes L_\EE$ as functors from $\EC$ to $(\EM\boxtimes Z(\EE))_{L_\EE}^0$. If $\EC=\EE'|_\EM$ in addition, then $\EM \boxtimes_\EE^{(\iota_\EM, \iota_0)} Z(\EE) \simeq (\EM, \iota_\EM)$.
\end{lem}
\begin{proof}
The second statement follows obviously from the first statement. 
To prove the identity $(\EM\boxtimes Z(\EE))_{L_\EE}^0 \simeq \EM$, by Eq.\,(\ref{eq:ZC_A}), it is enough to prove $Z((\EM \boxtimes Z(\EE))_{L_\EE})\simeq Z(\EM \boxtimes \EE)$ as braided fusion categories. Since $\EE$ is symmetric, for $m, x\in \EM, i\in \EE$, the action $(m\boxtimes i) \ast x:= m\otimes x \otimes i$ defines a left $\EM\boxtimes \EE$-module structure on $\EM$. It is enough to show that $(\EM \boxtimes Z(\EE))_{L_\EE}\simeq \Fun_{\EM\boxtimes \EE}(\EM, \EM)$ as UFC's \cite{eno2008,eno2009}, where the category $\Fun_{\EM\boxtimes \EE}(\EM, \EM)$ can be identified with the relative center $Z_\EE(\EM)$, and $Z(Z_\EE(\EM)) \simeq Z(\EM) \boxtimes Z(\EE)$ \cite{dgno1}. 

Consider the following commutative diagram: 
\be \label{diag:MZE}
\xymatrix{  \EM \boxtimes \EE \ar@{^{(}->}[rr]^{\id_\EM\boxtimes \iota_0} \ar[d]_\otimes & & \EM \boxtimes  Z(\EE)  \ar@{^{(}->}[r] \ar[d]_{\alpha}  &
\EM \boxtimes \overline{\EM} \boxtimes Z(\EE)  
\ar[dl]^{\hspace{0.5cm} { \mbox{\small$\alpha$-induction}}} \\
\EM \ar[rr]^{\hspace{-1cm}\alpha|_\EM: \,\,m\mapsto m\otimes^+ -} & & \Fun_{\EM\boxtimes \EE}(\EM, \EM) & ,
} 
\ee
where $m\in \EM$, the functor $\alpha$ is defined by $m\boxtimes j \mapsto
m\otimes^+ - \otimes^- j$ for $j=(j,z_{j,-}) \in Z(\EE)$, $z_{j,-}$ is the half-braiding, and the $\EM\boxtimes \EE$-module functor $m\otimes^+ - \otimes^- j \in 
%\EM \boxtimes \overline{\EM} \boxtimes Z(\EE)\to 
\Fun_{\EM\boxtimes \EE}(\EM, \EM)$ is defined by the following isomorphisms:
$$
m\otimes^+ x \otimes y \otimes i \otimes^- j \xrightarrow[\simeq]{c_{m,x} \otimes \id_y \otimes z_{j,i}^{-1}}  x\otimes m \otimes y \otimes j \otimes i, \quad\quad 
\forall x\in \EM, i\in \EE.
$$
The commutativity of the right triangle in (\ref{diag:MZE}) is nothing but the definition of the $\alpha$-induction. This implies that $m\otimes^+-\otimes^-j$ is a central functor. 
The commutativity of left square follows from the fact that there is a
canonical isomorphism between two $\EM\boxtimes \EE$-module functors
$m\otimes^+ - \otimes^- j \simeq m\otimes^+ j \otimes^+ -$, defined by the half braiding $z_{j,-}^{-1}: -\otimes^- j \to j \otimes^+ -$. Let $\alpha^\vee$ be the right adjoint functor of $\alpha$. 

\smallskip
We claim that $\alpha^\vee(\id_\EM) \simeq L_\EE$ as algebras. Indeed, we have 
$$
\hom_{\Fun_{\EM\boxtimes \EE}(\EM, \EM)}(m\otimes^+ - \otimes^- j, \id_\EM) 
= \hom_{Z_\EE(\EM)} (m\otimes^+ \one_\EM \otimes^- j, \one_\EM) \simeq \hom_{\EM\boxtimes Z(\EE)}(m\boxtimes j, \alpha^\vee(\id_\EM)).
$$ 
Without losing generality, we assume that $m$ and $j$ are both simple. Let $\sigma_{j,-}$ be the symmetric braiding in $\EE$. Since $\hom_{Z_\EE(\EM)} (m\otimes^+ \one_\EM \otimes^- j, \one_\EM) \hookrightarrow \hom_\EM(m\otimes j, \one_\EM)$, it is clear that $m\simeq j^\ast$ is a necessary condition for $\hom_{Z_\EE(\EM)} (m\otimes^+ \one_\EM \otimes^- j, \one_\EM)\neq 0$, but it is not sufficient. When $m\simeq j^\ast$, any morphism $\hom_\EM(m\otimes j, \one_\EM)$ is proportional to the canonical evaluation map $\ev: j^\ast \boxtimes j \to \one_\EM$. It is easy to check that $\ev$ is a morphism in $Z_\EE(\EM)$ (preserving the half-braiding) iff $(j, z_{j,-})\simeq(j, \sigma_{j,-}) \in \EE$. 
In summary, we obtain
$$
\hom_{Z_\EE(\EM)} (m\otimes^+ \one_\EM \otimes^- j, \one_\EM) =
\begin{cases}
\Cb & \mbox{if $m\simeq j^\ast$ and $(j, z_{j,-})\simeq(j, \sigma_{j,-}) \in \EE \subset Z(\EE)$},  \\
0 & \mbox{otherwise},
 \end{cases}
$$
which further implies that $\alpha^\vee(\id_\EM)\in \EM \boxtimes \EE$ and
$\alpha^\vee(\id_\EM)\simeq \oplus_{i\in O(\EE)} i\boxtimes i^\ast=L_\EE$ as
objects. To show $\alpha^\vee(\id_\EM) \simeq L_\EE$ as algebras, we use the commutative square in Diagram\,(\ref{diag:MZE}). It is enough to show that $(\alpha|_\EM)^\vee(\id_\EM) = \one_\EM$. Note that the functor $m\otimes^+-: \EM \to \Fun_{\EM\boxtimes \EE}(\EM, \EM)$ factors through the forgetful functor $f: Z(\EM) \to \Fun_{\EM\boxtimes \EE}(\EM, \EM)$. We must have 
$$
\one_\EM \hookrightarrow (\alpha|_\EM)^\vee(\id_\EM) = f^\vee(\id_\EM)\cap (\EM\boxtimes \one_{\overline{\EM}}) \hookrightarrow L_\EM \cap (\EM\boxtimes \one_{\overline{\EM}}) = \one_\EM.
$$ 
Therefore, $\alpha^\vee(\id_\EM) \simeq L_\EE$ as algebras.

By Thm.\,\ref{thm:dmno-1}, the category $(\EM \boxtimes Z(\EE))_{L_\EE}$ is monoidally isomorphic to a fusion subcategory of $\Fun_{\EM\boxtimes \EE}(\EM, \EM)$. By checking the Frobenius-Perron dimensions, we obtain that the functor $\alpha^\vee: \Fun_{\EM\boxtimes \EE}(\EM, \EM) \to (\EM \boxtimes Z(\EE))_{L_\EE}$ is a monoidal equivalence. Therefore, there is a canonical braided equivalence $g: \EM \xrightarrow{\simeq} (\EM \boxtimes Z(\EE))_{L_\EE}^0$, induced by the universal property of the Drinfeld center, such that the middle square in the following diagram 
%$$\xymatrix{\EM \boxtimes Z(\EE) \ar[r] \ar[rd] & (\EM \boxtimes Z(\EE))_{L_\EE}  & (\EM \boxtimes Z(\EE))_{L_\EE}^0 \ar@{_{(}->}[l] & (\EE \boxtimes \EE)_{L_\EE}^0 \ar@{_{(}->}[l] \\& \Fun_{\EM\boxtimes \EE}(\EM, \EM) \ar[u]_{\alpha^\vee}^\simeq & \EM \ar[l] \ar[u]^g & \EE \ar[l] \ar[u]}$$
$$
\xymatrix{ 
(\EM \boxtimes \EE)_{L_\EE}^0 \ar@{^{(}->}[rr] & & (\EM \boxtimes Z(\EE))_{L_\EE}^0 \ar@{^{(}->}[r] & (\EM \boxtimes Z(\EE))_{L_\EE}  & \EM \boxtimes Z(\EE) \ar[l]_{-\otimes L_\EE} \ar[ld]^\alpha \\
\EE'|_\EM  \ar[u]^{(- \boxtimes \one_\EE)\otimes L_\EE}_\simeq & \EC \ar@{_{(}->}[l]  \ar@{^{(}->}[r]^{\iota_\EM} & \EM \ar[r]^{\hspace{-1cm}m\mapsto m\otimes^- -} \ar[u]^g_\simeq & \Fun_{\EM\boxtimes \EE}(\EM, \EM)  \ar[u]^{\alpha^\vee}_\simeq  & 
}
$$
is commutative.

\smallskip
It remains to prove that the left square in above diagram is commutative. Note that the commutativity of the triangle is obvious. Since the functor $g$ is induced by the universal property of Drinfeld center, it is enough to prove that $\alpha^\vee \circ (m\mapsto m\otimes^- -) \circ \iota_\EM \simeq (-\boxtimes \one_\EE)\boxtimes L_\EE$, which further follows from the identities
$\alpha^\vee( c\otimes^- -) \simeq \alpha^\vee( -\otimes^+ c) \simeq \alpha^\vee( -\otimes^- c) \simeq \alpha^\vee(c\otimes^+ -) \simeq \alpha^\vee(\alpha(c \boxtimes \one_{Z(\EE)})) \simeq (c \boxtimes \one_{Z(\EE)})\otimes L_\EE$ for $c\in \EC$.
\end{proof}

As a special case, we obtain the following corollary. 
\begin{cor} \label{rema:unit-of-group} {\rm
Let $(\EM, \iota_\EM)$ be a modular extension of $\EE$. We have 
$\EM \boxtimes_\EE^{(\iota_\EM, \iota_0)} Z(\EE) \simeq (\EM, \iota_\EM)$.
}
\end{cor}

In summary, we have proved the first main result of this work.

\begin{thm} \label{thm:group}
The set $\EM_{ext}(\EE)$ of equivalence classes of the modular extensions of $\EE$, together with the binary multiplication $\boxtimes_\EE^{(-,-)}$ and the identity element $(Z(\EE),\iota_0)$, defines a finite abelian group. 
\end{thm}

%\begin{rema} {\rm The same results hold for the groups $\EM_{ext}(\EE)$ of unitary modular extensions. In general, $\EM_{ext}^u(\EE)\neq \EM_{ext}(\EE)$ (see Sec.\,\ref{sec:mext-svect}). }\end{rema}

\subsection{Modular extensions of $\rep(G)$ and group cohomologies} 
\label{sec:mext-repG}

%We set $\EE=\rep(G)$ throughout this section. 
Let $(\EM,\iota_\EM)$ be a modular extension of $\rep(G)$. The algebra $A=\text{Fun}(G)$ is a condensable algebra in $\rep(G)$ and also a condensable algebra in $\EM$. Moreover, $A$ is a Lagrangian algebra in $\EM$ because $(\dim A)^2 = (\dim \rep(G))^2 = \dim \EM$. Therefore, $\EM\simeq Z(\EM_A)$, where $\EM_A$ is the category of right $A$-modules in $\EM$. Moreover, the fusion category $\EM_A$ is pointed and equipped with a canonical faithful $G$-grading \cite{dgno1, dgno2,gnn}, which means that 
$$
\EM_A=\oplus_{g\in G} (\EM_A)_g, \quad (\EM_A)_g\simeq \vect, \,\, \forall g\in G, \quad \mbox{and} \quad \otimes: (\EM_A)_g \boxtimes (\EM_A)_h \xrightarrow{\simeq} (\EM_A)_{gh}.
$$ 

Let us recall the construction of this $G$-grading. Note that the functor
$F=-\otimes A: \EM \to \EM_A$ is a central functor. Namely, there is a
half-braiding $z_{m,x}: F(m) \otimes_A x \xrightarrow{} x \otimes_A F(m)$ for
$m\in \EM$. Let $x$ be a simple object in $\EM_A$. For $e\in \rep(G)$, $F(e)$
is a multiple of the tensor unit in $\EM_A$. Using the half-braiding, we obtain an isomorphism %(see also Figure\,\ref{fig:G-grading})
\be \label{eq:i-x}
F(e) \otimes_A x \xrightarrow{z_{e,x}} x \otimes_A F(e) = F(e) \otimes_A x, 
\ee
which is natural and monoidal with respect to the variable $e\in \rep(G)$.
Since $x$ is simple, we have $\aut(F(e)\otimes_A x) \simeq \aut(F(e))$. Thus, we obtain an automorphism of $F(e)$ that is natural and monoidal with respect to the variable $e\in \rep(G)$. In other words, 
we obtain a monoidal automorphism $\phi(x)$ of the fiber functor $F\circ \iota_\EM: \rep(G) \to \vect$. Therefore, we obtain a map $\phi: O(\EM_A) \to G$ defined by $x \mapsto \phi(x) \in \aut(F\circ \iota_\EM) = G$. Moreover, $\phi$ respects the multiplications and units. Furthermore, the non-degeneracy of $\EM$ implies that $\phi$ is a group isomorphism \cite{dgno2}. This defines a faithful $G$-grading on $\EM_A$. 

\begin{rema} {\rm
The physical meaning of acquiring a $G$-grading on $\EM_A$ after condensing the algebra $A=\fun(G)$ in $\EM$ is explained in \cite[Fig.\,1]{lkw2}. In fact, this is just a special case of a more general result, which says that the 2-category of non-degenerate braided fusion category containing $\rep(G)$ as a fusion subcategory is equivalent to the 2-category of G-crossed braided fusion categories via the functor $\EM \to \EM_A$ \cite{kirillov, mueger3,dgno2,gnn}. 
}\end{rema}

Since $\phi$ is an isomorphism, the associator of the monoidal category $\EM_A$ determines a unique $\omega_{(\EM,\iota_\EM)}\in H^3(G, U(1))$ such that $\EM_A \simeq \vect_G^{\omega_{(\EM,\iota_\EM)}}$ as $G$-graded unitary fusion categories and $\EM\simeq Z(\vect_G^{\omega_{(\EM,\iota_\EM)}})$ as UBFC's.

\void{
\begin{figure}[tb]
$$
%\raisebox{-3.5em}{
\begin{picture}(100, 185)
   \put(-60,0){\scalebox{1.45}{\includegraphics{pic-half-braiding.eps}}}
   \put(0,0){
     \setlength{\unitlength}{.75pt}\put(-18,-40){
     %\put(51, 148)     {  $x $}
     \put(20, 257)     { $e\in \rep(G) \subset \EM$}
     \put(-20, 170)    { $-\otimes A$}
     \put(75, 81)     { $x$}
     \put(75,57)        { $\gamma_2$}
     \put(75,132)        { $\gamma_1$}
     \put(100, 210)    {$\EM$}
     \put(225, 81)    {$\EM_A$}
     \put(160,45)     {$\vect$}
     \put(0, 65)        {$F(e)$}
     }\setlength{\unitlength}{1pt}}
  \end{picture}
  %}
$$
\caption{Consider a physical situation in which the excitations in the $2+1$D bulk are given by a modular extension $\EM$ of $\rep(G)$, and those on the gapped boundary by the UFC $\EM_A$. Consider a simple particle $e\in \rep(G)$ in the bulk moving toward the boundary. The bulk-to-boundary map is given by the central functor $-\otimes A: \EM \to \EM_A$, which restricted to $\rep(G)$ is nothing but the forgetful functor $F:\rep(G) \to \vect$. Let $x$ be a simple excitation in $\EM_A$ sitting next to $F(e)$. We move $F(e)$ along the semicircle $\gamma_1$ (defined by the half-braiding), then move along the semicircle $\gamma_2$ (defined by the symmetric braiding in the trivial phase $\vect$). This process defines the isomorphism in Eq.\,(\ref{eq:i-x}), which further gives a monoidal automorphism $\phi(x)\in \aut(F)=G$ of the fiber functor $F: \rep(G) \to \vect$. }
\label{fig:G-grading}
\end{figure}
}

\smallskip
Conversely, for any $\omega\in H^3(G, U(1))$, there is a canonical braided embedding $\iota_\omega: \rep(G) \hookrightarrow Z(\vect_G^\omega)$ such that the composition $\rep(G) \hookrightarrow Z(\vect_G^\omega) \to \vect_G^\omega$ defines a symmetric fiber functor $\rep(G) \to \vect\subset \vect_G^\omega$ and the induced group isomorphism $\phi: G=O(\vect_G^\omega) \to G$ is the identity map, i.e. $\phi=\id_G$ \cite{kirillov, mueger3,dgno2,gnn}.

\begin{thm} \label{thm:spt}
The map $(\EM, \iota_\EM) \mapsto \omega_{(\EM, \iota_\EM)}$ defines a group isomorphism $\EM_{ext}(\rep(G)) \simeq H^3(G, U(1))$. In particular, we have 
\be \label{eq:w+w}
Z(\vect_G^{\omega_1}) \boxtimes_{\rep(G)}^{(\iota_{\omega_1}, \iota_{\omega_2})} Z(\vect_G^{\omega_2}) \simeq (Z(\vect_G^{\omega_1+\omega_2}), \iota_{\omega_1+\omega_2}).
\ee 
\end{thm}
\pf
It is well-known that $Z(\rep(G))=Z(\vect_G)$. When $(\EM, \iota_\EM)=(Z(\vect_G), \iota_0)$, the forgetful functor $Z(\vect_G) \to \vect_G$ determines a unique Lagrangian algebra $A$ which is exactly the algebra $\fun(G)$ (recall Prop.\,\ref{expl:G-H}). Therefore, we have $\omega_{(Z(\rep(G)), \iota_0)} = 1$. 
It remains to prove the identity (\ref{eq:w+w}). Notice that $Z(\vect_G^{\omega_1}) \boxtimes Z(\vect_G^{\omega_2}) \simeq Z(\vect_{G\times G}^{\omega_1 \times \omega_2})$ as braided fusion categories, and we have $L_{\rep(G)}=\fun(G\times G/G)$ in $Z(\vect_{G\times G}^{\omega_1 \times \omega_2})$ by Remark\,\ref{rema:E-alg}. Then the identity (\ref{eq:w+w}) follows as a special case of Prop.\,\ref{expl:G-H} for the finite group 
$G\times G$ and its diagonal subgroup $G$. 
\epf

%\begin{rema} {\rm By Example\,\ref{expl:unitary}, for $\omega\in H^3(G,U(1))$, the MTC $Z(\vect_G^\omega)$ is automatically unitary.   }\end{rema}

\begin{rema} {\rm
Thm.\,\ref{thm:spt} matches precisely with the well-known group cohomology classification of bosonic SPT orders \cite{cglw}. Note that breaking the $G\times G$-symmetry on $Z(\vect_G^{\omega_1}) \boxtimes Z(\vect_G^{\omega_2})$ to the $G$-symmetry on $Z(\vect_G^{\omega_1}) \boxtimes_{\rep(G)}^{(\iota_{\omega_1}, \iota_{\omega_2})} Z(\vect_G^{\omega_2})$ exactly corresponds to condensing the algebra $L_{\rep(G)}=\fun(G\times G/G)$. 
}
\end{rema}

\begin{rema} \label{rema:dmitri} {\rm
Note that it is possible that $Z(\vect_G^{\omega_1})$ is braided equivalent to
$Z(\vect_G^{\omega_2})$ for $\omega_1\neq \omega_2$. For example, when
$G=\Zb_p$ for a prime number $p$, $H^3(\Zb_p,U(1))=\Zb_p$. But, the number of
monoidally non-equivalent fusion categories $\vect_{\Zb_p}^\omega$ is two for
$p=2$ and always three for $p\geq 3$ \cite{nik}, which is less than the number of
different 3-cocycles when $p\geq 5$. So the embedding $\rep(G) \hookrightarrow Z(\vect_G^\omega)$
is very important physical data that allows us to distinguish elements in the group $\EM_{ext}(\rep(G))$ as different bosonic SPT orders. 
}
\end{rema}

\subsection{Modular extensions of $\svect$ and Kitaev's 16-fold way} \label{sec:mext-svect}

In this subsection, we discuss a well-known classification of the modular extensions of the SFC $\svect$ (\cite{kitaev, dgno2}). 

\medskip
The symmetric fusion category $\svect$ contains two non-isomorphic simple objects: the tensor unit $\one$ and $u$ with $u\otimes u=\one$. The braiding $c_{u,u}\in \aut(u\otimes u) = \Cb^\times$ is $-1 \in \Cb^\times$. It can be viewed as the category $\rep(\Zb_2, z)$ of the representations of the group $(\Zb_2, z)$, where $z\in \Zb_2$ is the fermionic parity transformation, with the braiding $c_{u,u}$ defined above. 

\medskip
If $\EM$ is a modular extension of $\svect$, it is necessary that $\fpdim(\EM)=4$. We start with modular extensions of $\svect$ that are not pointed. Such modular extensions are called unitary Ising modular categories, each of which is a UMTC containing 3 equivalence classes of simple objects: the tensor unit $\one$, an invertible object $u$ and a non-invertible object $x$ with the following fusion rules: 
$$
u \otimes u \simeq \one, \quad u \otimes x \simeq x \simeq x \otimes u, 
\quad x\otimes x \simeq \one \oplus u, 
$$
and $\fpdim(u)=1$, $\fpdim(x)=\sqrt{2}$. The complete classification of such categories was obtained in \cite{vec} (see also \cite{fgv}), and was rediscovered more recently in \cite{kitaev,dgno2}. We will describe this classification following the labeling convention in \cite{dgno2}.

There are precisely two inequivalent monoidal structures (see for example \cite[Prop.\,B.5]{dgno2}). Each one has 4 different braided structures \cite[Cor.\,B.13]{dgno2}, which are automatically non-degenerate \cite[Cor.\,B.12]{dgno2}. Each of these 8 braided monoidal structures is determined uniquely by the braiding isomorphism $x\otimes x \to x\otimes x$ defined by $\zeta\, \id_\one \oplus \zeta^{-3}\, \id_u$ for $\zeta^8=-1$ \cite[Prop.\,B.14]{dgno2}. Each of the 8 has two spherical structures \cite[Sec.\,B.2]{dgno2} labeled by $\epsilon=\pm 1$. Therefore, there are 16 Ising modular categories. Due to the relation $\dim(x)=\epsilon(\zeta^2+\zeta^{-2})$, where $\dim(x)$ is the quantum dimension of $x$,
for each $\zeta$, only one of $\epsilon=\pm 1$ makes the Ising modular categories unitary. 
We denote the 8 UMTC's by $\EI_\zeta$. The S-matrix of $\EI_\zeta$ is given by (see \cite[Cor.\,B.21]{dgno2})
$$
S = \left( \begin{array}{ccc}      
1 & 1 & \sqrt{2} \\
1 & 1 & -\sqrt{2} \\
\sqrt{2} & -\sqrt{2} & 0
\end{array} \right)\, ,
$$
and the twists in $\EI_\zeta$ are given by (see \cite[Prop.\,B.20]{dgno2})
$$
\theta_{\one} = 1, \quad \theta_u = -1, \quad \theta_{X}=\epsilon\zeta^{-1},
$$
where $\epsilon=1$ if $\zeta^2+\zeta^{-2}=\sqrt{2}$ and $\epsilon=-1$ if $\zeta^2+\zeta^{-2}=-\sqrt{2}$. We would like to remark that $\theta_X$ for these 8 UMTC's are all distinct. % and take values in 8-th roots of $-1$. 

Each UMTC $\EI_\zeta$ contains a symmetric fusion subcategory that is generated by $\one$ and $u$ and is equivalent to $\svect$. Therefore, each $\EI_\zeta$ is a modular extension of $\svect$.

\medskip
If a modular extension $\EM$ of $\svect$ is pointed, then the group $G=O(\EM)$ must be abelian and of order $4$, and $\EM$ is equipped with a fully-faithful $G$-grading, i.e. $\EM=\oplus_{g\in G} \EM_g$ and $\EM_g\simeq \vect$. Let $x$ be the simple object in $\EM_g$, we define $q(g) := c_{x,x} \in \aut(x\otimes x) = \Cb^\times$. Then $q$ defines a non-degenerate quadratic form $q: G\to \Cb^\times$. Such a pair $(G,q)$ is called a metric group. The modular extension $\EM$ of $\svect$ is uniquely (up to isomorphisms) determined by the data $(G,q,u)$, where $(G,q)$ is a metric group of order 4 and $u$ is the order 2 element in $O(\svect)\subset G$ such that $q(u)=-1$. There are again 8 such modular extensions of $\svect$ \cite{kitaev}\cite[Example\,A.10, Lemma\,A.11]{dgno2}. More explicitly, these 8 modular extensions can be labeled by the set of 8-th roots of unity $\{ \kappa \in \Cb | \kappa^8=1\}$. Let $n(\kappa)=0$ if $\kappa^4=1$ and $n(\kappa)=1$ if $\kappa^4=-1$. Then the metric group $(G_\kappa, q_\kappa)$ associated to $\kappa$ is given by 
$$
G_\kappa :=\{ 0, v, u, v+u \, |\, 2u=0, 2v=n(\kappa)u \}, 
$$
and the quadratic form $q_\kappa$ is given by:
$$
q_\kappa(u)=-1, \quad q_\kappa(v)=q_\kappa(u+v)=\kappa, \quad q_\kappa(0)=1. 
$$
The twists in the associated modular tensor category $\EC(G_\kappa, q_\kappa)$ are $\theta_g=q_\kappa(g)$ for $g\in G_\kappa$, and the S-matrix of $\EC(G_\kappa, q_\kappa)$ is given by $S_{gh}=b(g,h)$, where $b(g,h)=\frac{q_\kappa(g+h)}{q_\kappa(g)q_\kappa(h)}$ for $g,h\in G_\kappa$. By Example\,\ref{expl:unitary}, these 8 modular tensor categories $\EC(G_\kappa,q)$ are all unitary. 

\medskip
In summary, we have the following result. 
\begin{thm}[\cite{kitaev,dgno2}] \label{thm:dgno2-ising-Aq}
There are 16 different modular extensions of $\svect$. They are given by 8 unitary Ising braided modular tensor categories $\EI_\zeta$ for $\zeta^8=-1$ and 8 unitary modular tensor categories $\EC(G_\kappa,q)$ associated to the metric group $(G_\kappa,q)$ for $\kappa^8=1$. 
\end{thm}

These 16 different modular extensions of $\svect$ are all different as non-degenerate braided fusion categories. Namely, the set of the modular extensions of $\svect$ coincides with that of the non-degenerate extensions of $\svect$. Moreover, these 16 non-degenerate extensions belong to 16 different Witt classes \cite{dgno2,dmno}. Note that the UMTC's $\EM\boxtimes \EN$ and $\EM \boxtimes_{\svect}^{(\iota_\EM, \iota_\EN)} \EN$ are Witt equivalent. Let $\EW$ be the Witt group. By taking Witt equivalence classes, we obtain an injective group homomorphism $[-]: \EM_{ext}(\svect) \hookrightarrow \EW$. It is well-known that the image is the subgroup $\Zb_{16}$ of $\EW$ \cite{dgno2,dmno,dno}. We obtain the following result. 

\begin{thm}
Taking Witt equivalence classes $[-]: \EM_{ext}(\svect) \simeq \Zb_{16}$ defines a group isomorphism. 
\end{thm}

Another convenient way to characterize the group $\EM_{ext}(\svect)$, especially for physicists, is to use the so-called multiplicative central charge. Recall that the {\it Gauss sums} of a pre-modular category $\EC$ are defined by 
$$
\tau^\pm (\EC) = \sum_{x\in O(\EC)} \theta_x^{\pm 1} \dim(x)^2, 
$$
where $\theta_x\in \aut(x) = \Cb^\times$ is the twist isomorphism. The so-called {\it multiplicative central charge} $\xi(\EC)$ is defined by $\xi(\EC):=\tau^+(\EC)/\sqrt{\dim(\EC)}$. 
It is well known that $\xi(\EC)$ is a root of unity. For modular tensor categories $\EC$ and $\ED$, it is known that 
$$
\xi(\EC\boxtimes \ED) = \xi(\EC)\xi(\ED), \quad\quad \xi(\overline{\EC}) \simeq \xi(\EC)^{-1}, 
$$
and $\xi: \EW \to \mathbb{Q}/8\Zb$ is a group homomophism \cite{dmno}. The multiplicative central charge defines a group isomorphism $\xi: \EM_{ext}(\svect) \xrightarrow{\simeq} \Zb_{16}$. The additive central charge $c=c(\EC)\in \mathbb{Q}/8\mathbb{Z}$ is related to $\xi(\EC)$ by $\xi(\EC)=e^{2\pi i c/8}$. Among all 16 modular extensions of $\svect$, the famous UMTC of the modules over the Ising Virasoro vertex operator algebra with additive central charge $c=1/2$ is mapped to $e^{2\pi i/16}$. It describes a $p+\mathrm{i} p$ superconductor state.

\begin{rema} {\rm
The relation between the modular extensions of $\svect$ and the classification of 2+1D topological superconductor is well known from the very beginning as Kitaev's 16 fold way \cite{kitaev}. The Witt classes of these 16 modular extensions form the $\Zb_{16}$ group is well-known \cite{dmno,dno}. Note also that 15 of the 16 are anisotropic in the sense that they can not be further condensed \cite{dmno}, thus can all be obtained by first stacking any one of them repeatedly then making the maximal condensations \cite{dmno} (see also a physical discussion of this fact in a recent paper \cite{nhksb}). But realizing the group $\Zb_{16}$ by the set $\EM_{ext}(\svect)$, together with the multiplication $\boxtimes_\EE^{(-,-)}$ and the identity element $(Z(\svect), \iota_0)$, is a new result. 
}
\end{rema}

\section{Modular extensions of \mtce's}

In this section, we study the relation between the sets of modular extensions of different \mtce's. We assume that all sets of modular extensions appeared in this section are not empty. 

\subsection{The set $\EM_{ext}(\EC)$ as a $\EM_{ext}(\EE)$-torsor} \label{sec:torsor}

In the simplest case, $\EC=\ED\boxtimes \EE$ and $\ED$ is a \umtc. Then $\EC$ is a \mtce.
In this case, the set $\EM_{ext}(\EC)$ of modular extension of $\EC$ is isomorphic to $\EM_{ext}(\EE)$. 

\smallskip
Let $\EC$ be a \mtce ~that has modular extensions. In general, there is no natural group structure on the set $\EM_{ext}(\EC)$. But there is a natural $\EM_{ext}(\EE)$-action on $\EM_{ext}(\EC)$: 
$$
\boxtimes_\EE^{(-,-)}: \EM_{ext}(\EC) \times \EM_{ext}(\EE) \to \EM_{ext}(\EC\boxtimes_\EE \EE) = \EM_{ext}(\EC)
$$  
by Prop.\,\ref{prop:associativity} and Remark.\,\ref{rema:unit-of-group}. 

\void{
\begin{prop}
Let $\EM$ be a modular extension of $\EC$ and $\EK$ that of $\EE$. 
We have the following results: 
\bnu
\item the pair $\left( (\EM\boxtimes \EK)_{L_\EE}^0, -\otimes L_\EE: \EC\to (\EM\boxtimes \EK)_{L_\EE}^0\right)$ gives a modular extension of $\EC$.
\item $\EM\boxtimes_\EE^{(\iota_\EM, \iota_0)} Z(\EE):=\left( (\EM \boxtimes Z(\EE))_{L_\EE}^0, -\otimes L_\EE\right) \simeq (\EM, \iota_\EM)$ as modular extensions of $\EC$. 
\item the action $\boxtimes_\EE^{(-, -)}: \EM_{ext}(\EC) \times \EM_{ext}(\EE) \to \EM_{ext}(\EC)$ is associative. 
\enu
\end{prop}
\pf
By Lemma\,\ref{lem:cclc}, the functor $-\otimes L_\EE: \EC\to (\EM\boxtimes \EK)_{L_\EE}^0$ is braided monoidal and fully faithful. It is clear that $\EC \subset \EE'|_{(\EM\boxtimes \EK)_{L_\EE}^0}$. By Eq.\,(\ref{eq:fpdim-centralizer}), we obtain that $\EC=\EE'|_{(\EM\boxtimes \EK)_{L_\EE}^0}$. This proves Part 1. Part 2 has been proved already (see Remark\,\ref{rema:unit-of-group}). The associativity of the actions follows from $\otimes \circ (\otimes \boxtimes \id_\EE)=\otimes \circ (\id_\EC \boxtimes \otimes): \EC \boxtimes \overline{\EC} \boxtimes \EE \to \EC$ (recall the proof of Prop.\,\ref{prop:associativity}). 
\epf
}

\begin{lem}
There is a map $\boxtimes_\EC^{(-,-)}: \EM_{ext}(\EC) \times \EM_{ext}(\overline{\EC}) \to \EM_{ext}(\EE)$ defined by 
$$
\left( (\EM, \iota_\EM), (\EN, \iota_\EN) \right) \mapsto \EM \boxtimes_\EC^{(\iota_\EM,\iota_\EN)} \EN :=
\left( (\EM \boxtimes \EN)_{L_\EC}^0, \,\, -\otimes L_\EC: \EE \to (\EM \boxtimes \overline{\EN})_{L_\EC}^0 \right), 
$$
and we have $\EM\boxtimes_\EC^{(\iota_\EM, \overline{\iota_\EM})} \overline{\EM} = (Z(\EE), \iota_0)$ (recall Lemma\,\ref{lem:cze}). 
\end{lem}
\pf
Let $(\EM, \iota_\EM)$ and $(\EN, \iota_\EN)$ be two modular extensions of a \mtce ~$\EC$. By Lemma\,\ref{lem:cclc}, the functor $-\otimes L_\EC: \EE \to (\EM \boxtimes \overline{\EN})_{L_\EC}^0$ is braided monoidal and fully faithful. Clearly, $\EE$ is a fusion subcategory of $\EE'|_{(\EM \boxtimes \overline{\EN})_{L_\EC}^0}$. By Eq.\,(\ref{eq:fpdim-centralizer}), we obtain 
$$
\fpdim(\EE)\, \fpdim(\EE'|_{(\EM \boxtimes \overline{\EN})_{L_\EC}^0}) =
\frac{\fpdim(\EM)\, \fpdim(\EN)}{\fpdim(L_\EC)^2} = \fpdim(\EE)^2. 
$$
As a consequence, we must have $\EE = \EE'|_{(\EM \boxtimes \overline{\EN})_{L_\EC}^0}$, i.e. $\left( (\EM \boxtimes \overline{\EN})_{L_\EC}^0, -\otimes L_\EC \right) \in \EM_{ext}(\EE)$. 
\epf

\begin{rema}{\rm
Note that there is an obvious isomorphism $\EM_{ext}(\EC) \xrightarrow{\simeq}  \EM_{ext}(\overline{\EC})$ defined by $(\EM, \iota_\EM) \mapsto (\overline{\EM}, \overline{\iota_\EM})$. 
}
\end{rema}

\begin{lem}
We have the following commutative diagram: 
\be \label{diag:boxtimes-C-E}
\xymatrix{
\EM_{ext}(\EC) \times \EM_{ext}(\overline{\EC}) \times \EM_{ext}(\EE)  \ar[rrr]^{\id_{\EM_{ext}(\EC)} \, \times \, \boxtimes_\EE^{(-,-)}}  
\ar[d]_{\boxtimes_\EC^{(-,-)}\, \times \, \id_{\EM_{ext}(\EE)}} & & &
\EM_{ext}(\EC) \times \EM_{ext}(\overline{\EC}) \ar[d]^{\boxtimes_\EC^{(-,-)}} \\
\EM_{ext}(\EE) \times \EM_{ext}(\EE) \ar[rrr]^{\boxtimes_\EE^{(-,-)}} & & & \EM_{ext}(\EE)\, .
}
\ee
\void{
and 
\be \label{diag:boxtimes-C-E-C}
\xymatrix{
\EM_{ext}(\EC) \times \EM_{ext}(\overline{\EC}) \times \EM_{ext}(\EC)  \ar[rrr]^{\id_{\EM_{ext}(\EC)} \, \times \, \boxtimes_\EE^{(-,-)}}  
\ar[d]_{\boxtimes_\EC^{(-,-)}\, \times \, \id_{\EM_{ext}(\EC)}} & & &
\EM_{ext}(\EC) \times \EM_{ext}(\overline{\EC}\boxtimes_\EE \EC) \ar[d]^{\boxtimes_\EC^{(-,-)}} \\
\EM_{ext}(\EE) \times \EM_{ext}(\EC) \ar[rrr]^{\boxtimes_\EE^{(-,-)}} & & & \EM_{ext}(\EC)\, .
}
\ee
}
\end{lem}
\pf
Let $(\EM, \iota_\EM), (\EN, \iota_\EN) \in \EM_{ext}(\EC)$ and $(\EP,\iota_\EP) \in \EM_{ext}(\EE)$. Then we have
\begin{align} 
(\EM\boxtimes_\EC^{(\iota_\EM,\overline{\iota_\EN})} \overline{\EN}) \boxtimes_\EE^{(\iota_\EM\boxtimes_\EC \overline{\iota_\EN}, \iota_\EP)} \EP &= 
\left( \left( (\EM \boxtimes \overline{\EN})_{L_\EC}^0 \boxtimes \EP\right)_{L_\EE}^0, \,\, \EE\hookrightarrow%{(\one_\EM \boxtimes \one_{\overline{\EN}} \boxtimes -) \otimes A_1} 
\left(\EM \boxtimes \overline{\EN})_{L_\EC}^0 \boxtimes \EP\right)_{L_\EE}^0 \right),  \label{eq:A1} \\
\EM\boxtimes_\EC^{(\iota_\EM, \,\, \overline{\iota_\EN}\boxtimes_\EE \iota_\EP)} (\overline{\EN}\boxtimes_\EE^{(\overline{\iota_\EN}, \iota_\EP)} \EP) &= \left( \left( \EM \boxtimes (\overline{\EN} \boxtimes \EP)_{L_\EE}^0\right)_{L_\EC}^0, \,\, \EE
\hookrightarrow %{(- \boxtimes \one_{\overline{\EN}} \boxtimes \one_\EP) \otimes A_2} 
\left( \EM \boxtimes (\overline{\EN} \boxtimes \EP)_{L_\EE}^0\right)_{L_\EC}^0 \right). \label{eq:A2}
\end{align}
First, notice that there are condensable algebras $A_1$ and $A_2$ in $\EM\boxtimes \overline{\EN} \boxtimes \EP$ such that 
$$
\left( (\EM \boxtimes \overline{\EN})_{L_\EC}^0 \boxtimes \EP\right)_{L_\EE}^0 \simeq 
(\EM\boxtimes \overline{\EN} \boxtimes \EP)_{A_1}^0, \quad 
\mbox{and} \quad
\left( \EM \boxtimes (\overline{\EN} \boxtimes \EP)_{L_\EE}^0\right)_{L_\EC}^0 \simeq
(\EM\boxtimes \overline{\EN} \boxtimes \EP)_{A_2}^0.
$$
The algebra $A_1$ can be uniquely (up to isomorphisms) determined by the image of the tensor unit under the following composed forgetful functors:
$$
\left( (\EM \boxtimes \overline{\EN})_{L_\EC} \boxtimes \EP\right)_{L_\EE} 
\xrightarrow{forget} (\EM \boxtimes \overline{\EN})_{L_\EC} \boxtimes \EP 
\xrightarrow{forget} \EM\boxtimes \overline{\EN} \boxtimes \EP.
$$
Instead of $A_1$, let us first consider how to determine $A_1\cap (\EC\boxtimes \overline{\EC} \boxtimes \EE)$. Restricting to the fusion subcategory $\EC\boxtimes \overline{\EC} \boxtimes \EE$ of $\EM\boxtimes \overline{\EN} \boxtimes \EP$, the right adjoint functors of above two forgetful functors give the left and the bottom functors, respectively, in the following diagram 
\be \label{diag:CCE}
\xymatrix{ \EC \boxtimes \overline{\EC} \boxtimes \EE \ar[rr]^-{\id_\EC \boxtimes \otimes}  \ar[d]_{\otimes \boxtimes \id_\EE} & & \EC \boxtimes \overline{\EC} \ar[d]^{\otimes} \\ \EC \boxtimes \EE \ar[rr]^-\otimes & & \EC\,. }
\ee
Therefore, we obtain 
\begin{align}
A_1 \cap (\EC\boxtimes \overline{\EC} \boxtimes \EE) &\simeq 
(\otimes \circ (\otimes \boxtimes \id_\EE))^\vee(\one_\EE) 
\simeq \oplus_{i\in \ob(\EE)} \otimes^\vee(i)\boxtimes i^\ast  \nn
&\simeq \oplus_{i\in \ob(\EE)} 
\left( L_\EC\otimes (\one_\EC \boxtimes i)\right) \boxtimes i^\ast  
\simeq (L_\EC \boxtimes \one_\EE) \otimes (\one_\EC \boxtimes L_\EE), 
\end{align}
where, in the third ``$\simeq$'', we have used the identity $\otimes^\vee(i)=L_\EC \otimes (\one_\EC \boxtimes i)$ (see for example \cite[Eq.\,(2.41)]{kong-runkel}). Since $\fpdim(A)=\fpdim(L_\EC)\fpdim(L_\EE)$, we must have 
\be \label{eq:A1-2}
A_1= (\otimes \circ (\otimes \boxtimes \id_\EE))^\vee(\one_\EC) = (L_\EC \boxtimes \one_\EE) \otimes (\one_\EC \boxtimes L_\EE).
\ee
Using similar arguments, we can show that
$$
A_2= (\otimes \circ (\id_\EC \boxtimes \otimes))^\vee(\one_\EC) = (\one_\EE \boxtimes L_\EE) \otimes (L_\EC \boxtimes \one_\EE).
$$ 
By the commutativity of the diagram (\ref{diag:CCE}), we must have $A_1\simeq A_2$ as algebras. It remains to prove that two embeddings of $\EE$ in Eq.\,(\ref{eq:A1}) and Eq.\,(\ref{eq:A2}) are isomorphic if we identify their codomains via $A_1\simeq A_2$. Note that these two embeddings can be identified with the functors $(\one_\EM\boxtimes \one_{\overline{\EN}} \boxtimes -) \otimes A_1$ and $(-\boxtimes \one_{\overline{\EN}} \boxtimes \one_\EP) \otimes A_2$, respectively. We have, for $x\in \EE$, 
$$
(\one_\EM \boxtimes \one_{\overline{\EM}} \boxtimes x)\otimes A_1  \simeq (\one_\EM \boxtimes x \boxtimes \one_\EP)\otimes A_1 
\simeq (x \boxtimes \one_{\overline{\EM}} \boxtimes \one_\EN) \otimes A_1 \simeq (x \boxtimes \one_{\overline{\EM}} \boxtimes \one_\EN) \otimes A_2. 
$$
Then it is clear that the functors $(\one_\EM\boxtimes \one_{\overline{\EN}} \boxtimes -) \otimes A_1$ and $(-\boxtimes \one_{\overline{\EN}} \boxtimes \one_\EP) \otimes A_2$ are isomorphic.
\epf

We are ready to state and prove the second main result of this work. 
\begin{thm} \label{thm:torsor}
The $\EM_{ext}(\EE)$-action on $\EM_{ext}(\EC)$ is free and transitive. In other words, the set $\EM_{ext}(\EC)$ is an $\EM_{ext}(\EE)$-torsor. 
\end{thm}
\pf
That the action is free follows from the identities, 
$$
\overline{\EM} \boxtimes_\EC^{(\overline{\iota_\EM}, \,\, \iota_\EM\boxtimes_\EE\iota_\EK)} (\EM\boxtimes_\EE^{(\iota_\EM,\,\, \iota_\EK)} \EK)\simeq 
(\overline{\EM} \boxtimes_\EC^{(\overline{\iota_\EM}, \,\, \iota_\EM)} \EM) \boxtimes_\EE^{(\overline{\iota_\EM}\boxtimes_\EE\iota_\EM, \,\,\iota_\EK)} \EK \simeq Z(\EE)\boxtimes_\EE^{(\iota_0, \iota_\EK)} \EK \simeq \EK,
$$ 
where the first $\simeq$ follows from the commutativity of the diagram (\ref{diag:boxtimes-C-E}) and the second $\simeq$ follows from Eq.\,(\ref{eq:MM=1}) for $(\EM,\iota_\EM) \in \EM_{ext}(\EC)$, $(\EK,\iota_\EK) \in \EM_{ext}(\EE)$.

To prove the transitivity of the $\EM_{ext}(\EE)$-action, we use a fundamental result \cite[Thm.\,5.20]{ffrs}, which says that the categories of local modules over two algebras in a MTC are canonically braided equivalent if these two algebras are the left and the right center of the same special symmetric Frobenius algebra, respectively. More explicitly, for any $(\EM, \iota_\EM), (\EN, \iota_\EN)\in \EM_{ext}(\EC)$, we define $(\EK, \iota_\EK): = \overline{\EM}\boxtimes_\EC^{(\overline{\iota_\EM}, \iota_\EN)} \EN  \in \EM_{ext}(\EE)$. It is enough to show that $\EN \simeq \EM \boxtimes_\EE^{(\iota_\EM, \iota_\EK)} \EK$. By Eq.\,(\ref{eq:MM=1}) and Lemma\,\ref{lem:unit-of-group}, it is enough to show that 
\be  \label{eq:transitive}
(\EM\boxtimes_\EC^{(\iota_\EM,\,\,\overline{\iota_\EM})} \overline{\EM}) \boxtimes_\EE^{(\iota_\EM\boxtimes_\EC \overline{\iota_\EM},\,\, \iota_\EN)} \EN \simeq 
\EM\boxtimes_\EE^{(\iota_\EM, \,\, \overline{\iota_\EM}\boxtimes_\EC \iota_\EN)} (\overline{\EM}\boxtimes_\EC^{(\overline{\iota_\EM},\,\, \iota_\EN)} \EN). 
\ee
More explicitly, using similar arguments used in proving Eq.\,(\ref{eq:A1-2}), we obtain
\begin{align}
(\EM\boxtimes_\EC^{(\iota_\EM,\overline{\iota_\EM})} \overline{\EM}) \boxtimes_\EE^{(\iota_\EM\boxtimes_\EC \overline{\iota_\EM}, \iota_\EN)} \EN &=\left( (\EM \boxtimes \overline{\EM} \boxtimes \EN)_{A_1}^0, \,\, f_1: \EC\xrightarrow{(\one_\EM \boxtimes \one_{\overline{\EM}} \boxtimes -) \otimes A_1} (\EM \boxtimes \overline{\EM} \boxtimes \EN)_{A_1}^0 \right),  \nn
\EM\boxtimes_\EE^{(\iota_\EM, \,\, \overline{\iota_\EM}\boxtimes_\EC \iota_\EN)} (\overline{\EM}\boxtimes_\EC^{(\overline{\iota_\EM}, \iota_\EN)} \EN) &= \left( (\EM \boxtimes \overline{\EM} \boxtimes \EN)_{A_2}^0, \,\, f_2: \EC\xrightarrow{(- \boxtimes \one_{\overline{\EM}} \boxtimes \one_\EN) \otimes A_2} (\EM \boxtimes \overline{\EM} \boxtimes \EN)_{A_2}^0  \right), \nonumber
\end{align}
where 
$$
A_1=  (L_\EC\boxtimes \one_\EN) \otimes (\one_\EM \boxtimes L_\EE), \quad\quad 
A_2= (L_\EE \boxtimes \one_\EN) \otimes (\one_\EM \boxtimes L_\EC)
$$
are condensable algebras in $\EM\boxtimes \overline{\EM} \boxtimes \EN$. 
It is a direct check that the condensable algebras $A_1$ and $A_2$ are the right center and the left center (\cite{oz}) of the algebra $A=(L_\EC \boxtimes \one_\EN) \otimes (\one_\EM \boxtimes L_\EC)$, respectively. The algebra $A$ is connected and separable but not commutative, and is automatically a symmetric special Frobenius algebra in the sense of \cite{ffrs} (see \cite[Remark\,2.8]{kong}). By \cite[Thm\,5.20]{ffrs}, there is a canonical composed braided equivalence 
$$
h: \,\, (\EM\boxtimes \overline{\EM}\boxtimes \EN)_{A_1}^0 \xrightarrow{\simeq} (\EM\boxtimes \overline{\EM}\boxtimes \EN)_{A|A}^0 \xrightarrow{\simeq} (\EM\boxtimes \overline{\EM}\boxtimes \EN)_{A_1}^0,
$$
where $(\EM\boxtimes \overline{\EM}\boxtimes \EN)_{A|A}^0$ is a well-defined full subcategory of the category of $A$-$A$-bimodules in $(\EM\boxtimes \overline{\EM}\boxtimes \EN)$ (see \cite[Def.\,5.6]{ffrs} for the precise definition). Moreover, by Eq.\,(5.38), (5.46) in \cite{ffrs} and the definition of the functor $G$ in the proof of Theorem\, 5.20 in \cite{ffrs}, the functor $h$ maps as follows
$$
(\one_\EM \boxtimes \one_{\overline{\EM}} \boxtimes c)\otimes A_1  \mapsto (\one_\EM \boxtimes \one_{\overline{\EM}} \boxtimes c)\otimes A 
\simeq (c \boxtimes \one_{\overline{\EM}} \boxtimes \one_\EN) \otimes A  \mapsto (c \boxtimes \one_{\overline{\EM}} \boxtimes \one_\EN) \otimes A_2,
$$
for $c\in \EC$. Then it is clear that $f_2\simeq h\circ f_1$. This completes the proof of the identity\,(\ref{eq:transitive}). 
\epf

\begin{rema} {\rm
Physically, the result above means that the difference of two symmetry enriched topological (SET) orders over a \umtce ~$\EC$ can be measured by SET orders over $\EE$, which are not unique in general (see Remark\,\ref{rema:eq-relation}). 
}
\end{rema}

\subsection{Symmetry breaking and group homomorphisms} \label{sec:sym-breaking}
 
Let $\EC$ be a \mtce, $\EM$ a modular extension of $\EC$ and $A$ a condensable algebra in $\EE$. The UMTC $\EM_A^0$ contains both categories $\EC_A$ and $\EE_A$ as fusion subcategories. It is clear that $\EC_A \subset \EE_A'|_{\EM_A^0}$. Moreover, we have $\fpdim(\EE_A)=\fpdim(\EE)/\fpdim{A}$ and $\fpdim(\EM_A^0)=\fpdim(\EC_A)\fpdim(\EE_A)$. Therefore, we must have $\EC_A=\EE_A'|_{\EM_A^0}$. Namely, $\EC_A$ is a $\mathrm{UMTC}_{/\EE_A}$ and $(\EM_A^0, \EC_A\hookrightarrow \EM_A^0)$ is a modular extension of $\EC_A$. Therefore, the assignment
$$
(\EM, \EC \hookrightarrow \EM) \mapsto (\EM_A^0, \EC_A \hookrightarrow \EM_A^0)
$$
defines a map $f_A: \EM_{ext}(\EC) \to \EM_{ext}(\EC_A)$ that describes a symmetry-breaking process.

\begin{rema} {\rm
When $\EE=\rep(G)$ and $A=\fun(G)$, we have $\EE_A\simeq \vect$ and $\EC_A=\EM_A^0$ is a UMTC.  
}
\end{rema}

\begin{prop}  \label{thm:sym-breaking-1}
When $\EC=\EE$, the map $f_A: \EM_{ext}(\EE) \to \EM_{ext}(\EE_A)$ is a group homomorphism. 
\end{prop}
\pf
We first prove that $f_A$ preserves the identity elements. Consider the following diagram: 
$$
\xymatrix{
Z(\EE) \ar[rr]^{-\otimes A} \ar@<+0.5ex>[d]^{\fr} & & Z(\EE)_A  \ar@<+0.5ex>[d]^{g} & &  Z(\EE)_A^0 \ar[d]^h \ar[ll]_{e_2}  \\
\EE \ar[rr]^{-\otimes A} \ar@<+0.5ex>[u]^{\iota_0}  & & \EE_A  \ar[urr]^{e_1} 
\ar@<+0.5ex>[u]^{e}  
\ar@<+0.5ex>[rr]^{\iota_0}  & & Z(\EE_A)  \ar@<+0.5ex>[ll]^{\fr}
}
$$
where $e,e_1,e_2$ are the canonical embeddings, and the functor $g$ is the restriction of the forgetful functor $fr: Z(\EE)\to \EE$ on $Z(\EE)_A$ because $Z(\EE)_A$ is naturally a subcategory of $Z(\EE)$. It is clear that the two overlapped left squares are commutative. 

We claim that the functor $g\circ e_2$ is a central functor. Indeed, if $(M, z_{M,-}) \in Z(\EE)$ is equipped with a local $A$-module structure, the half-braiding $z_{M,-}$ descend to a half braiding $\overline{z}_{M,-}$ on $Z(\EE)_A$, which further descends to a half-braiding $\overline{z}_{M,-}$ for the object $g\circ e_2(M,z_{M,-})=M$ in $\EE_A$. This half-braiding of $M\in \EE_A$ satisfies all the required properties of a central functor because $e_2$ is a central functor. 

Therefore, there is a unique braided functor $h: Z(\EE)_A^0 \to Z(\EE_A)$ such that $g\circ e_2=\fr\circ h$. Since both $Z(\EE)_A^0$ and $Z(\EE_A)$ are non-degenerate and have the same Frobenius-Perron dimensions, $h$ must be a braided equivalence. We claim that $h\circ e_1 = \iota_0$. This follows immediately from $\fr \circ h \circ e_1 = g\circ e_2\circ e_1=\id_{\EE_A}$ and the fact that such $h\circ e_1$ must be the unique lift of the central functor $\id_{\EE_A}$. We have proved that $(Z(\EE)_A^0, e_1) \simeq (Z(\EE_A), \iota_0)$ as modular extensions of $\EE_A$. Therefore, $f_A$ preserves the identity elements. 

It remains to prove that $f_A$ respects the multiplications. This amounts to show that, for any two modular extensions of $\EE$: $(\EM, \iota_\EM)$ and $(\EN, \iota_\EN)$, there is a braided equivalence such that the following diagram 
\be \label{diag:fA-multi}
\xymatrix{
((\EM \boxtimes \EN)_{L_\EE}^0)_A^0 \ar[rr]^{\simeq} & & (\EM_A^0 \boxtimes \EN_A^0)_{L_{\EE_A}}^0 \\
& \EE_A \ar[ul] \ar[ur]
}
\ee
is commutative. Let $R_1, R_2: \EE_A \to \EE \boxtimes \EE$ be the right adjoint functors of the two central monoidal functors $\EE \boxtimes \EE \to \EE_A$, respectively, in the following diagram:
$$
\xymatrix{
\EE \boxtimes \EE \ar[rr]^{\otimes} \ar[d]_{(-\otimes A) \boxtimes (-\otimes A)} & &  \EE \ar[d]^{-\otimes A}  \\
\EE_A \boxtimes \EE_A
 \ar[rr]^{\otimes} & & \EE_A\, .
}
$$
It is commutative because $-\otimes A$ is monoidal. Therefore, $R_1\simeq R_2$ and we have
$$
((\EM \boxtimes \EN)_{L_\EE}^0)_A^0 \simeq (\EM\boxtimes \EN)_{R_1(\one_{\EE_A})} \simeq (\EM\boxtimes \EN)_{R_2(\one_{\EE_A})} \simeq (\EM_A^0 \boxtimes \EN_A^0)_{L_{\EE_A}}^0. 
$$
The commutativity of the diagram (\ref{diag:fA-multi}) is tautological. 
\epf

\begin{expl}
Let $H$ be a subgroup of a finite group $G$. Let $\EE=\rep(G)$ and $A=\fun(G/H)$. We have $\EE_A=\rep(H)$. Recall that $\EM_{ext}(\rep(G))=H^3(G, U(1))$ and $\EM_{ext}(\rep(H))=H^3(H, U(1))$. The map $f_A: H^3(G, U(1)) \to H^3(H, U(1))$ in this case is just $\omega \mapsto \omega|_H$, which is clearly a group homomorphism. 
\end{expl}

\begin{prop}
We have the following commutative diagram: 
$$
\xymatrix{
\EM_{ext}(\EC) \times \EM_{ext}(\EE) \ar[rr]^{\boxtimes_\EE^{(-,-)}} \ar[d]_{f_A\times f_A}  & & \EM_{ext}(\EC) \ar[d]^{f_A}  \\
\EM_{ext}(\EC_A) \times \EM_{ext}(\EE_A) \ar[rr]^{\boxtimes_{\EE_A}^{(-,-)}}  & & \EM_{ext}(\EC_A)
}
$$
\end{prop}
\pf
It follows from the fact that the functor $-\otimes A: \EC \to \EC_A$ is monoidal, and the fact that the composed functor $\EC \boxtimes \EE \xrightarrow{\otimes} \EC \xrightarrow{-\otimes A} \EC_A$ is central. 
\epf

\subsection{Relation to Witt groups} \label{sec:witt}

In this subsection, we discuss the relation to Witt groups. We drop the assumption on the unitarity and consider the non-degenerate extensions of $\EE$. The unitary cases are similar. But note that the unitary Witt group is a proper subgroup of the usual Witt group (see \cite[Remark\,5.25]{dmno}).

\begin{defn} {\rm
A fusion category $\EA$ over $\EE$ is a fusion category equipped with a braided full embedding $T: \EE \to Z(\EA)$. 
}
\end{defn}

\begin{defn} {\rm
For a fusion category $\EA$ over $\EE$, the $/\EE$-center $Z_{/\EE}(\EA)$ of $\EA$ is defined by the M\"{u}ger centralizer of $\EE$ in $Z(\EA)$, i.e. $Z_{/\EE}(\EA):=\EE'|_{Z(\EA)}$. 
}
\end{defn}

\begin{defn} {\rm
Let $\EC$ and $\ED$ be two \nbfce's. $\EC$ and $\ED$ are called {\it Witt equivalent} if there exist fusion categories $\EA, \EB$ over $\EE$ and a braided $/\EE$-equivalence:
$$
\EC \boxtimes_\EE Z_{/\EE}(\EA) \simeq \ED \boxtimes_\EE Z_{/\EE}(\EB).
$$
}
\end{defn}

We denote the Witt class of $\EC$ by $[\EC]_{/\EE}$. If $\EE=\vect$, we simplify $[\EC]_{/\vect}$ to $[\EC]$. %; if $\EE=s\vect$, we simplify $[\EC]_{/s\vect}$ to $[\EC]_s$. 
We denote the set of Witt classes of \nbfce ~by $\EW_{/\EE}$. We simplify the notation $\EW_{/\vect}$ to $\EW$. 
% if $\EE=s\vect$, we simplify $\EW_{/s\vect}$ to $s\EW$. 

\begin{thm} {\rm \cite[Lem.\,5.2]{dno}}
The set $\EW_{/\EE}$ is an abelian group with the multiplication given by $\boxtimes_\EE$ and the identity element given by $[\EE]_{/\EE}$. 
\end{thm}

%Let $f: \EE \to \EF$ be a symmetric monoidal functor between two symmetric fusion categories $\EE$ and $\EF$. So $\EF$ is a braided fusion category over $\EE$. The assignment $\EC \mapsto \EC\boxtimes_\EE \EF$ defines a base change 2-functor from the 2-category of (braided) fusion categories over $\EE$ to that of (braided) fusion categories over $\EF$. 

\begin{lem} {\rm \cite[Prop.\,5.13]{dno}}
The assignment $[\EC] \mapsto [\EC\boxtimes \EE ]_{/\EE}$ is a well-defined group homomorphism from $\EW$ to $\EW_{/\EE}$. 
\end{lem}
\pf
Prop.\,5.13 in \cite{dno} was stated and proved only in the case $\EE=\svect$. But the same proof works for all $\EE$. For convenience of the readers, we include the proof here. It is enough to check that the map preserves the identity element and preserves the multiplication. If $\EC=Z(\EA)$ for some fusion category $\EA$, then $\EC\boxtimes \EE=\EE'|_{\EC\boxtimes Z(\EE)}$. Therefore, $\EC\boxtimes \EE = Z_{/\EE}(\EA\boxtimes \EE)$, i.e. $[\EC\boxtimes \EE]_{/\EE}=[\EE]_{/\EE}$. The identity $
[(\EC\boxtimes \EE)\boxtimes_\EE (\ED \boxtimes \EE)]_{/\EE} = [ \EC\boxtimes \ED \boxtimes \EE]_{/\EE}$ implies that the map preserves the multiplication.
\epf

We denote the group homomorphism by $[-\boxtimes \EE]_{/\EE}$. To relate modular extensions to Witt groups, we first generalize a result in \cite{dno}. 
\begin{prop} \label{lem:surjective}
The assignment $\EM \to [\EM]$ defines a surjective group homomorphism from $\EM_{ext}(\EE)$ to the kernel of the canonical group homomorphism $[-\boxtimes \EE]_{/\EE}: \EW \to \EW_{/\EE}$.
\end{prop}
\pf
If $\EM$ is a modular extension of $\EE$, we have a braided equivalence $\overline{\EM} \boxtimes \EM \simeq Z(\EM)$. Note that the canonical full embedding $\EE = \one_{\overline{\EM}}\boxtimes \EE \hookrightarrow Z(\EM)$ is braided monoidal. Therefore, $\EM$ is a fusion category over $\EE$. Since $\EE'|_{\EM}=\EE$, we obtain $\overline{\EM}\boxtimes \EE \simeq Z_{/\EE}(\EM)$. Therefore, $[\overline{\EM}]$ is in the kernel of $\EW \to \EW_{/\EE}$, and so is $[\EM]$. The map $[-]$ is clearly a group homomorphism. % because $[\EM\boxtimes_\EE^{(\iota_\EM, \iota_\EN)} \EN]^{-1}=[\EM\boxtimes \EN]^{-1}=[\EM]^{-1}[\EN]^{-1}$. 

To prove the surjectivity, consider a Witt class $[\EM]$ in the kernel of $[-\boxtimes \EE]_{/\EE}: \EW \to \EW_{/\EE}$. By definition, there is a fusion category $\EA$ over $\EE$ such that there is a braided $/\EE$-equivalence: $\EM \boxtimes \EE \simeq Z_{/\EE}(\EA)$. Note that $\EM$ is a fusion subcategory of $Z_{/\EE}(\EA)$ and, therefore, a fusion subcategory of $Z(\EA)$. Since both $\EM$ and $Z(\EA)$ are non-degenerate, we must have $Z(\EA) \simeq \EM \boxtimes \EB$, where $\EB$ is non-degenerate. Therefore, we must have a full embedding $\EE\simeq \EE''|_{Z(\EA)}= Z_{/\EE}(\EA)'|_{Z(\EA)} \hookrightarrow \EM'|_{Z(\EA)} \simeq \EB$. Moreover, we have $\EE'|_\EB \simeq \EM' \boxtimes \EE'|_\EB \simeq (\EM\boxtimes \EE)'|_{Z(\EA)} \simeq \EE''|_{Z(\EA)}\simeq\EE$. In other words, $\EB$ is a modular extension of $\EE$, so is $\overline{\EB}$. Notice that $[\EM]=[\overline{\EB}]$. Therefore, $\EM_{ext}(\EE)$ maps onto the kernel of $\EW \to \EW_{/\EE}$. 
\epf

\begin{cor}
The canonical group homomorphism $\EW \to \EW_{/\rep(G)}$, defined by $[\EC]\mapsto [\EC\boxtimes \rep(G)]_{/\rep(G)}$, is injective. 
\end{cor}
\pf
This follows immediately from Lemma\,\ref{lem:surjective} and Thm.\,\ref{thm:spt}. 
\epf

\begin{thm}[\cite{dgno2,dno}] \label{thm:dno-svect}
The map from the set $\EM_{ext}(\svect)$ to the kernel of the canonical group homomorphism $\EW \to \EW_{/\svect}$, defined by $\EC \mapsto [\EC]$, is bijective. 
\end{thm}
\pf 
By Prop\,\ref{lem:surjective}, it is enough to prove the injectivity, which was proved in \cite{dgno2,dmno}. 
\epf

\section{Conclusions and Outlooks} \label{sec:outlook}

In this work, we prove that the set $\EM_{ext}(\EE)$ of (the equivalence classes of) modular extensions of a symmetric fusion category $\EE$ is a finite abelian group, and the set $\EM_{ext}(\EC)$ of modular extensions of an \umtce ~$\EC$ is a $\EM_{ext}(\EE)$-torsor. We explain in details how these groups of modular extensions recover the well-known physical results of the group-cohomology classification of bosonic SPT orders and Kitaev's 16 fold way. 
We also explain briefly the behavior of these groups under symmetry-breaking processes. We hope to convince readers that there is a very rich physical and mathematical theory behind the scene, and we have only scratched its surface. There are many important problems left to be studied. We list a few open problems that are worth studying. 

\bnu

\item Explicitly identify the group $\EM_{ext}(\rep(G,z))$. Physically, we believe that this group should give the classification of SET orders over $\rep(G,z)$ up to $E_8$ quantum Hall states. The subgroup of $\EM_{ext}(\rep(G,z))$ consisting of modular extensions with central charge $c=0$ ($\mathrm{mod}\,\, 8$) classifies all the fermonic SPT orders with symmetry $(G,z)$. 

\item For a generic \mtce ~$\EC$, it is possible that there is no modular extension of $\EC$. Examples of such \mtce's are constructed by Drinfeld for certain integral \mtce ~$\EC$ with $\fpdim(\EC)=8$ and $\fpdim(\EE)=4$ \cite{drinfeld}. It is an important open problem to characterize those \mtce's that admit modular extensions. Its solutions should also deepen our understanding of its physical meaning. 

\item If the modular extension of a given \mtce ~$\EC$ does not exist, it means that the symmetry $\EE$ is anomalous (not on-site and not gaugable). We believe that this phenomenon is detectable by certain global structures that appear when we integrate the local data $\EC$ (defined on an open 2-disk) over all closed 2d surfaces via factorization homology \cite{bbj}. See \cite{lkw2} for more speculations on this issue. %Only result we know so far is in the no-symmetry case, i.e. $\EE=\vect$, in which no such inconsistency occurs and the factorization homology of $\EC$ was proved to be the category $\vect$, objects in which are the ground state degeneracies (or hom spaces in $\EC$) on closed surfaces with finite number particles \cite{ai}. 

\item If a \umtce ~$\EC$ does not have any (minimal) modular extension (see Remark\,\ref{rema:non-minimal}). We can always embed $\EC$ into some modular tensor categories, such as $Z(\EC)$, with higher Frobenius-Perron dimensions. What is the minimal Frobenius-Perron dimension of such non-minimal modular extensions? What are the physical meanings of these non-minimal modular extensions of $\EC$? Do they form any interesting mathematical structures for each fixed Frobenius-Perron dimension? %From a physical point of view, the solution to these problems should reveal some information of the 3+1D bulk theory with a gapped boundary given by $\EC$. 

\enu

\end{document}